\documentclass[11pt]{amsart}
\usepackage{quiver}
\usepackage{amsmath}
\usepackage{amssymb}
\usepackage{mathrsfs}
\usepackage[all]{xy}
\usepackage{graphicx}
\usepackage{latexsym}
\usepackage{amsxtra}
\usepackage{url}
\usepackage{verbatim}
\usepackage{adjustbox}

\usepackage{todonotes}
\setuptodonotes{inline,color=teal!20}

\usepackage{tikz}
\usetikzlibrary{shapes.misc, positioning}

\usepackage{float}

\usepackage[urlcolor=blue]{hyperref}

\newtheorem{theorem}{Theorem}[section]
\newtheorem{claim}[theorem]{Claim}

\newtheorem{lemma}[theorem]{Lemma}

\newtheorem{proposition}[theorem]{Proposition}
\newtheorem{corollary}[theorem]{Corollary}

\theoremstyle{definition}
\newtheorem{definition}[theorem]{Definition}

\newtheorem{question}[theorem]{Question}

\theoremstyle{remark}
\newtheorem{remark}[theorem]{Remark}

\newtheorem{fact}[theorem]{Fact}

\def\l{{\langle}}
\def\r{{\rangle}}

\newcount\skewfactor
\def\mathunderaccent#1#2 {\let\theaccent#1\skewfactor#2
\mathpalette\putaccentunder}
\def\putaccentunder#1#2{\oalign{$#1#2$\crcr\hidewidth
\vbox to.2ex{\hbox{$#1\skew\skewfactor\theaccent{}$}\vss}\hidewidth}}

\def\smallbox#1{\leavevmode\thinspace\hbox{\vrule\vtop{\vbox
   {\hrule\kern1pt\hbox{\vphantom{\tt/}\thinspace{\tt#1}\thinspace}}
   \kern1pt\hrule}\vrule}\thinspace}

\DeclareMathOperator{\cof}{cof}
\DeclareMathOperator{\spec}{Sp}
\DeclareMathOperator{\dom}{dom}
\DeclareMathOperator{\supp}{Supp}

\DeclareMathOperator{\rng}{rng}
\newcommand{\cf}{{\rm cf}}

\newcommand{\s}{\subseteq}

\title{Isomorphism classes of generating sets}
\author[Benhamou]{Tom Benhamou}
\address[Benhamou]{Department of Mathematics, Rutgers University, New Brunswick, NJ 08854, USA}
\thanks{The first author was supported by the NSF grant  DMS-2346680}
\email{tom.benhamou@rutgers.edu}

\author[Cummings]{James Cummings}
\address[Cummings]{Department of Mathematical Sciences, Carnegie Mellon University,
  Pittsburgh, PA 15213, USA}
  \email{jcumming@legba.math.cmu.edu}
\thanks{The second author was partially supported by NSF
grant DMS–2054532.}

\author[Goldberg]{Gabriel Goldberg}
\address[Goldberg]{Department of Mathematics, UC Berkeley, CA 94720, USA}
\email{ggoldberg@berkeley.edu}
\thanks{The third author was partially supported by NSF Foundations Grant 2401789. }

\author[Hayut]{Yair Hayut}
\address[Hayut]{Einstein Institute of Mathematics, Hebrew University of Jerusalem, Givat-Ram, 91904, Israel.}
\email{yair.hayut@mail.huji.ac.il}
\thanks{The forth author was partially supported by the Israel Science Fund, ISF 1967/21.}

\author[Poveda]{Alejandro Poveda}
\address[Poveda]{ Department of Mathematics and Center of Mathematical Sciences and Applications, Harvard University, Cambridge MA, 02138, USA}
\email{alejandro@cmsa.fas.harvard.edu}
\urladdr{www.alejandropovedaruzafa.com}
\thanks{The fifth author was supported by the the Center of Mathematical Sciences and Applications and the Department of Mathematics at Harvard University.}
\thanks{The results of this paper were conceived at the American Institute of Mathematics in Pasadena as part of a SQuaRE program. We thank the Institute for their support and for providing an excellent working environment.}
\begin{document}

\begin{abstract}
    We introduce a new class of ultrafilters which generalizes the well-known class of simple $P$-point ultrafilters. We prove that for any well-founded $\sigma$-directed partial order $\mathbb{D}$ there is a mild forcing extension where there is an ultrafilter $U$ on $\omega$ with a base $\mathcal{B}$ such that $(\mathcal{B},\supseteq^*)\cong \mathbb{D}$. 
    On a measurable cardinal we prove a similar result: relative to a supercompact cardinal, it is consistent that $\kappa$ is supercompact, and for a $\kappa^+$-directed well-founded poset $\mathbb{D}$, there is a ${<}\kappa$-directed closed $\kappa^+$-cc forcing extension where there is a \emph{normal} ultrafilter $U$ on $\kappa$ with a base $\mathcal{B}$ such that $(\mathcal{B},\supseteq^*)\cong \mathbb{D}$. These are optimal results in the class of $P$-points and realize every potential structure of a $P$-point. We apply our constructions to obtain ultrafilters with controlled Tukey-type, in particular, an ultrafilter with non-convex Tukey and depth spectra is presented, answering questions from \cite{TomCohesive}. Our construction also provides new models where $\mathfrak{u}_\kappa<2^\kappa$, answering questions from \cite{TomGabe2025}.
\end{abstract}
\maketitle
\section{Introduction}

Ultrafilters with strong combinatorial properties have proven useful in constructing mathematical objects exhibiting extreme behavior, such as ultraproducts, topological spaces (including special points within them), and certain groups. They can also be employed to establish purely combinatorial results \cite{KomjathBook,GoldbringUltrafilters,SetTheoreticTopology}. Many of these  results are obtained by constructing a special generating set for the ultrafilter. This paper aims to study the possible structures of such generating sets. More precisely, we are interested in the isomorphism class of the structure $(\mathcal{B},\supseteq^*)$, where $\supseteq^*$ is the reversed inclusion modulo the bounded ideal (or finite in case of $\omega$) and $\mathcal{B}$ is a $\subseteq^*$-generating set of a filter $F$, namely, $\mathcal{B}\subseteq F$ and for every $A\in F$, there is $B\in \mathcal{B}$ such that $B\subseteq^* A$.

For example, Kunen \cite{kunen} used an iteration of Mathias forcing relative to an ultrafilter from \cite{MathiasHappyFamily} to construct ultrafilters on $\omega$ with generating sets which are $\subseteq^*$-decreasing of any desired ordinal length of uncountable cofinality $\lambda$, these are called simple $P_\lambda$-points. Kunen used this to separate the ultrafilter number  $\mathfrak{u}$ from the continuum, where $\mathfrak{u}$ is defined to be the least cardinality of a generating set for a uniform ultrafilter on $\omega$. Our main construction generalizes Kunen's construction to form generating sets of more complex types. In particular, we prove the following consistency result:
\begin{theorem}\label{Main Thm 1}
    Let $\lambda_0<\lambda_1$ be cardinals of uncountable cofinality. Then there is a ($\lambda_1^+$-directed closed)$\ast$ccc forcing extension in which there is an ultrafilter $U$ on $\omega$, with a $\subseteq^*$-generating set $\mathcal{B}\subseteq U$ such that $(\mathcal{B},\supseteq^*)\simeq \lambda_0\times\lambda_1$.  
\end{theorem}
The proof uses non-linear iterations of Mathias forcing. Such iterations were considered before by Hechler \cite{Hechler} who proved a similar result to obtain dominating families in the structure $(\omega^\omega,\leq^*)$. Several results in the spirit of Hechler can be found in the literature \cite{Barnabas,BurkeKada,BARTOSZYNSKIKada}. Blass and Shelah \cite{BlassShelah} considered a non-linear iteration involving Mathias forcing, however, the Mathias component only really appears in one dimension of this construction. The major difficulty, however, is to obtain at the end of the process (i.e. iteration) a uniform ultrafilter. Milovich \cite[Thm. 5.10]{MilovichNeoterian}, produced via traditional linear iteration an ultrafilter $U$ on $\omega$ which is generated by a set $\mathcal{B}$ isomorphic to a cofinal subset of $\lambda_0\times\lambda_1$.

In an attempt to realize other orders, and generalizing Kunen's simple $P_\lambda$-point, we define:
\begin{definition}
    Given an ordered set $\mathbb{D}=(D,\leq_D)$, we say that a filter $F$ is a {\em simple  $P_{\mathbb{D}}$-point} if there is a $\subseteq^*$-generating set $\mathcal{B}\subseteq F$ such that $(\mathcal{B},\supseteq^*)\simeq \mathbb{D}$.
\end{definition} 
Hence Theorem \ref{Main Thm 1} says that in a forcing extension, there is a uniform simple $\mathbb{P}_{\lambda_0\times\lambda_1}$-point ultrafilter. Our result is much more general and applies to every well-founded $\sigma$-directed set.
\begin{theorem}\label{main thm 1}
    Let $\mathbb{D}$ be a well-founded, $\sigma$-directed set, which has a possible order type\footnote{$\mathbb{D}$ has $\lambda$ as a possible order type if there is a linear order $\prec$ on $D$ of order type $\lambda$, such that $<_\mathbb{D}\subseteq \prec$.}  $\lambda$. Then there is a ($\lambda^+$-directed closed)$\ast$ccc forcing extension in which there is a simple $P_\mathbb{D}$-point ultrafilter on $\omega$. 
\end{theorem}
Starting from any model, the $\lambda^+$-closed forcing in the previous theorem is simply the standard forcing to add a $\diamondsuit_{\lambda^+}(\lambda^+\cap \text{cof}(\lambda))$ sequence with initial segments. If this diamond already holds in the ground model, we can directly force with the ccc forcing. We also note that our method works equally well if $\diamondsuit_{\mu}(\mu\cap \text{cof}(\lambda))$ holds for any $\mu\geq\lambda^+$ regular.  Hence we have the following corollary:
\begin{corollary}\label{main cor 1}
    Let $\mathbb{D}$ be a well-founded, $\sigma$-directed set, which has a possible order type $\lambda$. Assume that $\diamondsuit_{\mu}(\mu\cap \text{cof}(\lambda))$ holds for some regular cardinal $\mu>\lambda$. There is a ccc forcing extension  in which $2^\omega = \mu$ and there is a simple $P_\mathbb{D}$-point ultrafilter on $\omega$. 
\end{corollary}

It is essential to note that for any given $\mathbb{D}$, a simple $\mathbb{P}_{\mathbb{D}}$-point can easily be forced if we do not require some nice properties of the forcing which is used to create it, so as to preserve the
properties of $\mathbb D$.
For example, we can always collapse $|\mathbb{D}|$ to $\omega_1$ and run Kunen's argument. Hence the strength  of Theorem \ref{main thm 1} and Corollary \ref{main cor 1} lies within the type of forcings we use. In Section \ref{section: omega case} we prove these theorems, and at the beginning of Section \ref{Section: Tukey order} we establish that the chain condition and closure of the forcing ensures the preservation of the structure of $\mathbb{D}$.

Generalizing Kunen's result to measurable cardinals requires considerable effort and does not run smoothly using the standard indestructibility methods. In an unpublished work, Woodin and Carlson showed that from a supercompact cardinal, this is possible. Recently, this type of iteration appeared in several constructions \cite{GSOnDO,5authors,BROOKETAYLOR201737,DzSh} guided by (for example) a diamond sequence so that a supercompact embedding lifts in a way that the normal ultrafilter derived from the lifted embedding has a $\subseteq^*$-decreasing generating set of the desired length. This can be used to obtain a model with $2^\kappa>\kappa^+$, with a $\kappa$-complete ultrafilter over $\kappa$ generated by fewer than $2^\kappa$-many sets. It remains open whether this can be achieved from weaker large cardinal assumptions. It is this type of argument which we use to prove Theorem \ref{main thm 1}.

Our method is flexible enough to incorporate lifting arguments and produce similar results on measurable cardinals:
\begin{theorem}\label{main theorem 2}
    Let $\kappa$ be a supercompact cardinal and $\mathbb{D}$ be a well-founded, $\kappa^+$-directed set with a possible order type $\lambda$.  Then for every  regular $\mu > \lambda$, there is an extension by a ($\kappa$-cc) $\ast$ ($\mu$-directed-closed) $\ast$ ($\kappa$-directed-closed $\kappa^+$-cc) forcing in which there is a normal ultrafilter which is a simple $P_{\mathbb{D}}$-point ultrafilter over $\kappa$.
\end{theorem}
The first forcing here is the the usual Laver preparation \cite{Laver1978} which makes the supercompactness of $\kappa$-indestructible under $\kappa$-directed closed forcings, and the second forcing adds 
a $\diamondsuit_\mu(\mu \cap \text{cof}(\lambda))$-sequence. 
Since the Laver preparation forcing
 is $\kappa$-cc with cardinality $\kappa$, it will preserve such a diamond sequence if it already existed. 
\begin{corollary}\label{main corollary 2}
Let $\kappa$ be a supercompact cardinal, let $\mathbb{D}$ be a well-founded, $\kappa^+$-directed set with a possible order type $\lambda$, and let $\mu > \lambda$ be a regular cardinal such that
$\diamondsuit_\mu(\mu \cap \text{cof}(\lambda))$ holds.
Then there is an extension by a ($\kappa$-cc)  $\ast$ ($\kappa$-directed-closed $\kappa^+$-cc) forcing in which there is a normal ultrafilter which is a simple $P_{\mathbb{D}}$-point ultrafilter over $\kappa$.
\end{corollary}
These results can also be used to construct complex generating sets for the club filter on a measurable cardinal as well:
\begin{corollary}\label{main corollary 3}
    Let $\kappa$ be a supercompact cardinal and $\mathbb{D}$ be a well-founded, $\kappa^+$-directed set with a possible order type $\lambda$.  Then for every  regular $\mu > \lambda$, there is an extension by a ($\kappa$-cc) $\ast$ ($\mu$-directed-closed) $\ast$ ($\kappa$-directed-closed $\kappa^+$-cc) forcing in which $\kappa$
    is measurable and the club filter on  $\kappa$ is a simple $P_{\mathbb{D}}$-point.
\end{corollary}
Theorem \ref{main theorem 2} and Corollaries \ref{main corollary 2},\ref{main corollary 3} will be proven in Section \ref{Section at a measur}.

All the ultrafilters which are produced in this paper are $P$-points. This is an integral feature of the Mathias iteration. Our results are optimal in the class of $P$-points since if a $P$-point over $\kappa\geq\omega$ is a simple $\mathbb{P}_{\mathbb{D}}$-point, then $\mathbb{D}$ has to be $\kappa^+$-directed. It is also evident that every order has a well-founded cofinal subset. So our results say that every such order can also be realized.

Finally, in Seciton \ref{Section: Tukey order}, we provide several applications of our construction.
The first is the ability to control the Tukey-type of certain ultrafilters. Recall that the Tukey order on directed sets is defined by $(\mathbb{P},\leq_P)\leq_T (\mathbb{Q},\leq_{Q})$ if there is a Tukey reduction $f:\mathbb{P}\to\mathbb{Q}$, that is, for every unbounded $\mathcal{B}\subseteq \mathbb{P}$, $f[\mathcal{B}]$ is unbounded in $\mathbb{Q}$. We say that $\mathbb{P}\equiv_T\mathbb{Q}$ if $\mathbb{P}\leq_T \mathbb{Q}$ and $\mathbb{Q}\leq_T\mathbb{P}$. In this paper, we will only consider ultrafilters $U$ ordered by reversed inclusion $(U,\supseteq)$ and almost inclusion $(U,\supseteq^*)$. The study of the Tukey order finds its origins in the concept of Moore-Smith convergence of nets and has been studied extensively on general ordered sets \cite{tukey,Isbell65,Solecki/Todorcevic04}. The Tukey-type $(U,\supseteq)$ also has been subject to a considerable amount of work, especially when $U$ is an ultrafilter on $\omega$ \cite{Isbell65,DobrinenJSL15,Dobrinen/Todorcevic11,Blass/Dobrinen/Raghavan15,Raghavan/Shelah17,Raghavan/Verner19,Milovich08,tomFanxin,Tomcommute} and lately this has been also looked at in the context of measurable cardinals \cite{TomNatasha,TomGabe24} due to its close relation to the Galvin property and applicability to Prikry-type forcing theory \cite{Parttwo,GitDensity,bgp}. The Tukey-type of $(U,\supseteq^*)$ as also been of interest \cite{Milovich08,TomNatasha} as the properties of a
$\subseteq^*$-generating set
      for an ultrafilter are related to this Tukey-type. In fact, most of the combinatorial characteristics of ultrafilters (see for example \cite{Brendle1997UltrafiltersO}) are formulated in terms of $\subseteq^*$ rather than $\subseteq$. The connection to the Tukey order is noticeable once the Tukey order is expressed in terms of cofinal maps. Given a partially ordered set $(D,\leq_D)$, a set $B\subseteq D$ is called cofinal in $(D,\leq_D)$ if for every $d\in D$ there is $b\in B$ such that $d\leq_D b$. Schmidt Duality \cite{Schmidt55} in the special case of ultrafilters (or if there are least upper bounds) asserts that:
For any two ultrafilters $U,W$, $U\leq_T W$ if and only if there is a function $f:W\to U$ which is monotone (i.e.\ $A\subseteq B\Rightarrow f(A)\subseteq f(B)$) and has a cofinal image in $U$.  

It is not hard to check that whenever $U$ is a simple $P_{\mathbb{D}}$-point, then $(U,\supseteq^*)\equiv_T \mathbb{D}$. So, the analysis in this paper is finer than the Tukey-type analysis of ultrafilters, but can be used to control it. For example, we use our construction to find an ultrafilter $U$ which is Tukey-equivalent to $\lambda_0\times\lambda_1$. From this we can analyze the Tukey spectrum and the depth spectrum (see Definitions \ref{Def: Tukey Spec}, \ \ref{Def: Depth Spec}) and show the consistency of an ultrafilter with a non-convex spectrum (both depth and Tukey). This was asked by the first author in \cite{TomCohesive}.
\begin{theorem} {}\
    \begin{enumerate}
        \item It is consistent that there is an ultrafilter $U$ on $\omega$ such that $\spec_T(U)$ and $\spec_{dp}(U)$ are non-convex sets.
        \item Relative to the existence of a supercompact cardinal, it is consistent that there is a normal ultrafilter $U$ such that $\spec_T(U)$ and $\spec_{dp}(U)$ are non-convex sets.
    \end{enumerate}
\end{theorem}

Our second application relates to generalized cardinal characteristics. Recall that
$$\mathfrak{b}_\kappa=\min\{|\mathcal{F}|\mid \mathcal{F}\subseteq\kappa^\kappa\text{ is }{\leq^*}\text{-unbounded}\}$$
$$\mathfrak{d}_\kappa=\min\{|\mathcal{F}|\mid \mathcal{F}\subseteq\kappa^\kappa\text{ is }{\leq^*}\text{-dominating}\}$$
$$\mathfrak{u}_\kappa=\min\{|\mathcal{B}|\mid \mathcal{B} \ {\subseteq^*}\text{-generates a uniform ultrafilter on }\kappa\}$$
$$\mathfrak{u}^{com}_\kappa=\min\{|\mathcal{B}|\mid \mathcal{B} \ {\subseteq^*}\text{-generates a }\kappa\text{-complete ultrafilter on }\kappa\}$$

As we mentioned before, at a measurable cardinal, the only previously known method to produce models where $\mathfrak{u}_\kappa<2^\kappa$ is via the linear Mathias iteration. The computation of the values of these cardinal characteristics was done in \cite{BROOKETAYLOR201737} where it was shown that after a linear iteration of Mathias forcing, $\mathfrak{u}_\kappa=\mathfrak{u}^{com}_\kappa=\mathfrak{b}_\kappa=\mathfrak{d}_\kappa$. Lately, the first and the third author \cite{TomGabe2025} generalized this and showed that merely the presence of a simple $P_\lambda$-point (which is the kind of ultrafilters produced by the Mathias iteration) causes $\mathfrak{u}_\kappa=\mathfrak{u}^{com}_\kappa=\mathfrak{b}_\kappa=\mathfrak{d}_\kappa=\lambda$. So the question of separating the cardinal characteristics in a model where $\mathfrak{u}_\kappa<2^\kappa$ remained open. Using our methods, we can prove the following:
\begin{theorem}
    Relative to the existence of a supercompact cardinal, it is consistent that $\kappa$ is measurable and  $\mathfrak{b}_\kappa<\mathfrak{d}_\kappa=\mathfrak{u}_\kappa<2^\kappa$. In particular, there is no simple $P$-point in that model.
\end{theorem}
\section{Some preliminaries}
Given a partial order $\mathbb{D}=(D,\leq_D)$, we define $d_1<_Dd_2$ if $d_1\leq_D d_2$ and $d_1\neq d_2$. For each \(d\in D\), let \(D_{<d} = \{e\in D : e <_D d\}\); \(D_{\leq d}\) is defined similarly.

In this paper, we would like to study what are the possible posets $\mathbb{D}$  for which there are uniform simple $P_{\mathbb{D}}$-point ultrafilters, both on $\omega$ and on measurable cardinals. For the following easy proposition, we define that a poset $\mathbb{P}=(P,\leq_P)$ embeds cofinally into a poset $\mathbb{Q}=(Q,\leq_Q)$ if there is an order preserving injection $j:\mathbb{P}\to \mathbb{Q}$ such that $\rng(j)$ is a cofinal subset of $\mathbb{Q}$.
\begin{proposition}\label{Prop: downwards type} Let $U$ be an ultrafilter.  If $U$ is a simple $P_{\mathbb{D}}$-point and $\mathbb{D}'$ is embedded cofinally in $\mathbb{D}$, then $U$ is a simple  $P_{\mathbb{D}'}$-point.
\end{proposition}
There are some trivial limitations on a poset $\mathbb{D}$  for which there can be a simple $P_\mathbb{D}$-point ultrafilter. For example, $\mathbb{D}$ has to be directed. Also, $\mathbb{D}$ cannot have maximal elements. The following limitation generalizes \cite[Prop. 1.4]{Brendle1997UltrafiltersO}.
\begin{theorem}
    Let $U$ be a uniform $\kappa$-complete ultrafilter over $\kappa\geq\omega$. If $U$ is a simple $P_{\mathbb{D}}$-point, then $\mathbb{D}$ is not the increasing union of $\kappa$-many $\kappa$-directed non-cofinal subsets of $\mathbb{D}$.
\end{theorem}
\begin{proof}
    Suppose otherwise, that $\mathbb{D}=(D,\leq_D)$ and let $D=\bigcup_{i<\kappa}D_i$, where $\l D_i\mid i<\kappa\r$ is an $\subseteq$-increasing sequence of $\kappa$-directed non-cofinal subsets of $D$. By moving to a subsequence of the $D_i$'s if necessary, we may assume that for every $i$, there is an element $e_i$ of $D_{i+1}$ which is not bounded by any element of $D_i$.  

    Let $\mathcal{B}=(b_d)_{d\in D}$ be a $\subseteq^*$-generating set for $U$ witnessing that $U$ is a simple $P_{\mathbb{D}}$-point. Let $F_i$ be the $\kappa$-complete filter generated by $(b_d)_{d\in D_i}$. Note that since $D_i$ is $\kappa$-directed,  $(b_d)_{d\in D_i}$ is a $\subseteq^*$-generating set for $F_i$. Let  $F=\bigcup_{i<\kappa}F_i$ and let us argue that $F=U$. Indeed, the filters $F_i$ are increasing (as the $D_i$'s are increasing) and therefore $F$ is a filter. Clearly,  $\mathcal{B}$, which is a generating set for $U$, is included in $F$ and by maximality of ultrafilters, $F=U$. 
    
    Recall that $e_i\in D_{i+1}$ is not bounded by any element of $D_i$ and consider $A_i=\bigcap_{j\leq i} b_{e_j}$. Then:
    \begin{claim}
        {}\
    \begin{enumerate}
     \item The sequence $\l A_i\mid i<\kappa\r$ is $\subseteq$-decreasing.
        \item $A_i\in F_r$ for every $r>i$.
        \item $A_i\notin F_i$
        \item $\bigcap_{i<\kappa} A_i\notin U$
    \end{enumerate}
    \end{claim}
    \begin{proof}[Proof of claim.]  $(1)$ is trivial. For $(2)$, by the choice of $e_j$, $e_j\in D_{j+1}$. Since the $D_r$'s are increasing, $e_j\in D_r$ for every $r>j$. Hence for every $j\leq i$ and every $r>i$, $b_{e_j}\in F_r$. By $\kappa$-completeness, $A_i=\bigcap_{j\leq i} b_{e_j}\in F_r$.  For $(3)$, suppose otherwise that $A_i\in F_i$. Since $(b_d)_{d\in D_i}$ generates $F_i$, we can find $d\in D_i$  such that $b_d\subseteq^* A_i$. But then $b_d\subseteq^* b_{e_i}$, and since $((b_d)_{d\in D},\supseteq^*)\simeq \mathbb{D}$ we conclude that $e_i\leq_D d\in D_i$. This contradicts the choice of $e_i$.
    
    $(4)$ follows easily from $(3)$ and the fact that $\bigcup_{i<\kappa}F_i=U$.   $\qedhere_{\text{claim}}$
    \end{proof}
    Consider the set $A^*=A_0\setminus\bigcap_{i<\kappa}A_i$. By $(4)$ of the claim, $A^*\in U$. Let
    
    $C_\alpha=(\bigcap_{\beta<\alpha}A_\beta)\setminus A_\alpha$. Using $(1)$, it is not hard to check that $A^*$ can be partitioned as $A^*=\biguplus_{1\leq\alpha<\kappa}C_\alpha$.
    
    Next, we split. Define $A^*$:
    $$X_{\text{even}}=\bigcup_{\alpha<\kappa} C_{2\alpha}, \ X_{\text{odd}}=\bigcup_{\alpha<\kappa}C_{2\alpha+1}$$
    Since $X_{\text{even}}\cup X_{\text{odd}}=A^*\in U$, either $X_{\text{even}}\in U$ or $X_{\text{odd}}\in U$. Suppose for example that $X_{\text{even}}\in U$ (the argument in the case that $X_{\text{odd}}\in U$ is identical). Then there is $i_1<\kappa$ such that $X_{\text{even}}\in F_{i_1}$. By $(2)$ of the claim, and $\kappa$-completeness of $F_{i_1}$, $\bigcap_{\alpha<i_1}A_\alpha\in F_{i_1}$. Therefore $X_{\text{even}}\cap\bigcap_{\alpha<i_1}A_\alpha\in F_{i_1}$. Since the $F_r$'s are increasing, we may assume that $i_1=2\gamma+1$ for some $\gamma<\kappa$. Finally, we claim that $X_{\text{even}}\cap\bigcap_{\alpha<2\gamma+1}A_\alpha\subseteq A_{2\gamma+1}$. This produces a contradiction since $X_{\text{even}}\cap\bigcap_{\alpha<2\gamma+1}A_\alpha\in F_{2\gamma+1}$ while $A_{2\gamma+1}\notin F_{2\gamma+1}$.
     Indeed if $\nu\in X_{\text{even}}\cap\bigcap_{\alpha<2\gamma+1}A_\alpha$ then by definition of $X_{\text{even}}$ there is $\alpha^*<\kappa$ such that $\nu\in C_{2\alpha^*}$. Since $\nu\in \bigcap_{\alpha<2\gamma+1}A_\alpha$, it follows that $2\alpha^*>2\gamma+1$. This means that $\nu\in A_{2\gamma+1}.$
\end{proof}
\begin{corollary}\label{cor: limitation}
    There is no uniform simple $P_{\omega\times \omega_1}$-point ultrafilter on $\omega$.
\end{corollary}
The ultrafilters we will be constructing are all $P$-points. This poses more restrictions on the possible $\mathbb{D}$ which we will be able to realize. 

Recall that an ultrafilter $U$ is a $P_\lambda$-point, if $(U,\supseteq^*)$ is  $\lambda$-directed.
\begin{proposition}
    $U$ is a $P_{\lambda}$-point iff for every (any) $\mathbb{D}$  such that $U$ is a simple $P_{\mathbb{D}}$-point, $\mathbb{D}$ is $\lambda$-directed.
\end{proposition}

\section{$\mathbb{D}$-based iteration} \label{Dbased}
In this section we will describe a forcing iteration which is indexed by a well-founded partial order. As we noted in the introduction, similar iterations appeared in various places in the literature. Notably, in the works of Shelah (see, for example, \cite{Shelah2000, Shelah2004}), this type of iteration is modeled as linear iteration with \emph{restricted memory} --- so in Shelah's language, asking the memory to be transitive is parallel to forcing along a partial order. In \cite{Shelah2004}, Shelah described a more general mechanism that allows one to deal with certain schemes which might be ill--founded or non-transitive.

We describe the general method with respect to ${<}\kappa$-supports where $\kappa$ is either $\aleph_0$ (so, in this case we are referring to \emph{finite support}) or measurable. As the treatment of the basic properties of the forcing in both cases is very similar, we include both of them in this section. 

Let us define the notion of ${<}\kappa$-support iteration based on $\mathbb{D}$ uniformly for every partial order $\mathbb{D}$.

Let \(\mathbb{D}=(D,\leq_D)\) be a partial order and let $\kappa$ be a regular cardinal, 
a sequence \(\langle \mathbb P_{<d},\dot{\mathbb Q}_d : d\in D\rangle\)
is a \textit{${<}\kappa$-support iteration based on \(\mathbb{D}\)} if the following conditions hold for all \(d\in D\):
\begin{itemize}
    \item \(\mathbb P_{<d}\) is the set of ${<}\kappa$-supported sequences 
    $p=\langle p_e\mid e\in D_{<d}\rangle$ such that for all \(e\in D_{<d}\), $\Vdash_{\mathbb{P}_{<e}}p(e)\in \dot{\mathbb{Q}}_e$.
    \item \(\Vdash_{\mathbb P_{<d}} \dot{\mathbb Q}_d \text{ is a partial order}\).
\end{itemize}
We denote by \(\mathbb P_{\mathbb{D}}\) 
the set of ${<}\kappa$-supported sequences 
$p=\langle p_e\mid e\in D\rangle$ such that for all \(e\in D\), $\Vdash_{\mathbb{P}_{<e}}p(e)\in \dot{\mathbb{Q}}_e$.
The order is defined recursively by $$p\leq_{\mathbb{P}_{\mathbb{D}}} q\text{ iff for every }d\in D, \  p\restriction D_{<d}\Vdash_{\mathbb{P}_{<d}}p(d)\leq_{\dot{\mathbb{Q}_d}}q(d).$$ 

We will also write
\(\mathbb P_{\leq d}\) instead of
\(\mathbb P_{<d} * \dot{\mathbb Q}_d\). 
Ultimately, we will only consider iterations based on well-founded partially ordered sets, as the names $\dot{\mathbb{Q}}_d$ will be defined by well-founded recursion.

\begin{lemma}\label{Lemma: projection by restriction}
    Let $\mathbb{D}'\s \mathbb{D}$ be downwards closed. Then $\mathbb{P}_{\mathbb D}$ projects via restriction to $\mathbb{P}_{\mathbb{D}'}.$
\end{lemma}
\begin{proof}
    We will check that the restriction map from \(\mathbb P_{\mathbb D}\) to \(\mathbb P_{\mathbb D'}\) is a projection.
    Clearly, it is order preserving. 
    Suppose that $q\in \mathbb{P}_{\mathbb{D}'}$ is stronger than $p\restriction D'$. 
    We will find an extension \(p'\leq p\) in \(\mathbb P_{\mathbb D}\) such that \(p'\restriction D' = q\).
    
    Define $p'=q\cup p\restriction (D\setminus D')$. 
    Then by definition, $p'$ belongs to \(\mathbb P_{\mathbb D}\), and $p'\restriction D'=q$. It remains to prove that for every $d\in D$, 
   
    $$p'\restriction D_{<d}\Vdash_{\mathbb{P}_{<d}}p'(d) \leq p(d).$$ 
    
    If $d\in D'$, since $D'$ is downward closed, $p'\restriction D_{<d}=q\restriction D_{<d}$ and $p'(d)=q(d)$. Since $q$ is stronger than $p\restriction D'$, we are done. 
    
    Suppose instead that $d\in D\setminus D'$. Then $p'(d)=p(d)$, and clearly, $p'\restriction D_{<d}\Vdash_{\mathbb{P}_{<d}} p'(d)\leq p(d)$.
\end{proof}
As one can see from the proof, this projection is a \emph{forking projection} \cite{SigmaPrikry}, that is to say given a condition $p$ and a condition $r \leq p \restriction D'$, there is a weakest condition $q = r \cup (p\restriction D\setminus D')$ which extends $p$ and project to $r$.

Recall that the \textit{Mathias forcing relative to a filter} is defined as follows:
\begin{definition}\label{Def:Mathias}
    Let $F$ be a $\kappa$-complete filter over a regular cardinal $\kappa\geq\omega$. Let $\mathbb{M}(F)$ be the forcing notion consisting of conditions of the form $(a,A)\in [\kappa]^{<\kappa}\times F$. The order is defined by $(a,A)\leq (b,B)$ if and only if
    \begin{enumerate}
        \item $b\sqsubseteq a$.
        \item $A\subseteq B$.
        \item $a\setminus b\subseteq B$
    \end{enumerate}
\end{definition}
This forcing is $\kappa$-closed, and $(\kappa^{<\kappa})^+$-cc. It adds a set $X$ which is a $\subseteq^*$-lower bound for the filter $F$. Note that $\mathbb{M}(F)$ is \textit{canonically} $\kappa$-directed closed, namely, every directed set of conditions of size ${<}\kappa$ has a greatest lower bound.  Indeed, for a directed set of conditions $\langle s_i, A_i\rangle$, their greatest lower bound is $\langle \bigcup s_i, \bigcap A_i\rangle$. 

\begin{lemma}\label{lemma:basic-properties}
Let $\mathbb{D}$ be well-founded, 
and let $\mathbb{P}_{\mathbb{D}}$ be ${<}\kappa$-support iteration based on $\mathbb{D}$ such that for every $d\in \mathbb{D}$, $\Vdash_{\mathbb{P}_{<d}} \dot{\mathbb{Q}}_d$ is the Mathias forcing relative to a $\kappa$-complete filter $\dot{F}_d$. Then:
\begin{enumerate}
    \item $\mathbb{P}_{\mathbb{D}}$ has a dense subset of $\mathbb{P}_{\mathbb{D}}$ of conditions $p$ such that for every $d\in \dom p$, $p(d) = \langle \check s, \dot{A}\rangle$ (namely, the stem is a canonical name).
    \item $\mathbb{P}_{\mathbb{D}}$ is canonically $\kappa$-directed closed.
    \item $\mathbb{P}_{\mathbb{D}}$ is $(\kappa^{<\kappa})^+$-Knaster and in particular $(\kappa^{<\kappa})^+$-cc.
\end{enumerate}
\end{lemma}

\begin{proof}  
    Let us prove item $(1)$ by induction on $\mathbb{D}_{+}= \mathbb{D} \cup \{\infty\}$, where $\infty$ is defined to be above all elements of $\mathbb{D}$. 
    
    Let $d\in \mathbb{D}_{+}$ and let us assume that for every $e < d$ the lemma holds for $\mathbb{P}_{<e}$.  Let $p \in \mathbb{P}_{<d}$, and let $s = \supp p = \langle e_i \mid i < i_*\rangle$, $i_* < \kappa$. Let us assume that we pick the enumeration so that $e_i <_{\mathbb{D}} e_j$ implies $i < j$.

    Define a sequence of conditions $\langle p_i \mid i < i_* + 1\rangle$ in $\mathbb{P}_{<d}$. Let $p_0 = p$ and for every $i$, let $p_{i+1} \leq p_i$ such that 
    \begin{itemize}
        \item    $p_{i+1} \restriction \mathbb{D}_{<e_i}$ decides the stem of $p_i(e_i)$, 
        \item for every $e' \in \dom p_{i+1} \cap \mathbb{D}_{<e_i}$, the stem of $p_{i+1}(e')$ is a canonical name for a ground model set, and 
        \item \(p_{i+1} \restriction (\mathbb{D}_{<d}\setminus \mathbb{D}_{<e_i}) = p_{i} \restriction (\mathbb{D}_{<d}\setminus \mathbb{D}_{<e_i}).\)
    \end{itemize}
    The second requirement is possible by the inductive hypothesis.

    For a non-zero limit ordinal $i \leq i_*$, let $\supp p_i = \bigcup_{j < i} \supp p_j$ and for every $e \in \supp p_i$, let $p_i(e)$ be the greatest lower bound of $\langle p_j(e) \mid j < i,\, e \in \dom p_j\rangle$. By the above remark, we can take the stem of $p_i(e)$  to be the canonical name for the union of the stems. 

    It is straightforward to verify that $p_i$ is indeed a condition which is stronger than $p_j$ for all $j < i$.
 Note that at step $i$, for every $e \in \supp p_i$ such that $e \neq e_k$ for $k \geq i$, the stem of $p_i(e)$ is a canonical name for a ground model set, as wanted.

For $(2)$, note that  for every $d\in\mathbb{D}$, $\dot{\mathbb Q}_d$ is forced by $\mathbb{P}_{<d}$ to be canonically $\kappa$-directed closed and so a canonical lower bound can be defined coordinate-wise.

Let us prove $(3)$.    Let $\mathcal{A} \subseteq \mathbb{P}_{\mathbb{D}}$ be a collection of conditions of cardinality $(\kappa^{<\kappa})^+$. Using the first part of the lemma, we may assume that for every $p\in \mathcal{A}$, and $e\in \supp p$, the stem of $p(e)$ is a canonical name.

    Since the support of the the conditions is of size $<\kappa$, there is $\mathcal{S}\subseteq \mathcal{A}$ of cardinality $(\kappa^{<\kappa})^+$ such that $\langle \supp p \mid p \in \mathcal{S}\rangle$ form a $\Delta$-system. Let $R$ be the root of the system. 
    By shrinking $\mathcal{S}$, we may assume that there is a function $s\colon R \to [\kappa]^{<\kappa}$ such that for every $p\in\mathcal{S}$, and $e\in R$, $p(e) = \langle \check{s(e)}, \dot{A}^p_e\rangle$. 
    We claim that every pair of conditions $p, q \in \mathcal{S}$ is compatible.  To see this
    let 
    \[r(e) = \begin{cases}
    p(e) & e \in \supp p \setminus R \\
    q(e) & e \in \supp q \setminus R \\
    \langle \check {s(e)}, \dot{A}^p_e \cap \dot{A}^q_e\rangle & e \in R.
    \end{cases}\]

As each $\dot{F}_e$ is forced to be a filter, this is a condition stronger than $p$ and $q$.
The same argument shows that any subset of $\mathcal{S}$ of cardinality less than $\kappa$ has a lower bound. 
\end{proof}

As we will see later, the main challenge is to pick the filters $\dot{F}_e$ in a way that will allow us to isolate in the generic extension an ultrafilter  which is generated by the Mathias generic sets. 

Given an iteration $\mathbb{P}_{\mathbb{D}}$ of Mathias forcing, a generic filter $G\subseteq \mathbb{P}_\mathbb{D}$, and $d\in \mathbb{D}$, let $$G_{<d} = \{p\restriction D_{<d} \mid p \in G\}.$$ By Lemma \ref{Lemma: projection by restriction}, it is a generic filter for $\mathbb{P}_{<d}$. 
Let $G_d = \{p(d)^{G_{<d}} \mid p \in G\} \subseteq \mathbb Q_d$, which is generic over $V[G_{<d}]$. Let $B_d = \bigcup \{s \mid \langle s, A\rangle \in G_d\}$  and fix a canonical name $\dot{B}_d$ for $B_d$.

\begin{lemma}\label{lemma:incompatibility}
    Let $\mathbb{P}_{\mathbb{D}}$ be a ${<}\kappa$-support iteration based on $\mathbb{D}$ of Mathias forcings relative to filters and let $G\subseteq \mathbb{P}_{\mathbb{D}}$ be a generic filter. In $V[G]$, for all $e, d\in \mathbb{D}$ such that $e\not\leq d$,  $B_d\nsubseteq^*  B_e$.
\end{lemma}
\begin{proof}
    Assume otherwise and let $p\in\mathbb{P}_{\mathbb{D}}$ force that $\dot{B}_d \setminus \dot{B}_e$ is bounded by some $\alpha < \kappa$. 
    Let us look at $p(d) = \langle \check s_d, \dot{A}_d\rangle$ and $p(e) = \langle \check s_e, \dot{A}_e\rangle$. 
    As $\dot{A}_d$ is a $\mathbb{P}_{<d}$-name for an unbounded subset of $\kappa$, there is an extension $r \leq p \restriction \mathbb{D}_{<d}$ in $\mathbb{P}_{<d}$ such that for some $\gamma > \alpha, \sup(s_e)$, \  $r\Vdash\gamma \in \dot{A}_d$.  

    Let $\dot{A}'_e$ be a $\mathbb{P}_{<e}$-name forced to be $\dot{A}_e\setminus\check{\gamma+1}$ and let $s'_d=s_d\cup\{\gamma\}$.
    Let $q\leq p$ be the following condition:
    \[q(c) = \begin{cases}
        r(c) & c < d \\
        \langle \check{s}'_d, \dot{A}_d\rangle & c = d \\
        \langle \check{s}_e, \dot{A}'_e\rangle & c = e \\
        p(c) & else.
    \end{cases}\]
    Then, $q \leq p$ and $q \Vdash \check\gamma \in \dot{B}_d \setminus \dot{B}_e$, contradiction.
\end{proof}

\begin{corollary}\label{cor:diagonalizing-the-ultrafilter}
    Let $\mathbb{P}_{\mathbb{D}}$ be as above and let  $\mathbb{E} \subseteq \mathbb{D}$ such that
    \begin{align}\label{condition}
     \text{for every }e_0,e_1 \in \mathbb{E}, \text{ if }e_0 <_E e_1\text{ then }\Vdash_{\mathbb{P}_{<e_1}} \dot{B}_{e_0} \in \dot{F}_{e_1}.\tag{$\dagger$}\end{align}
    Then, in the generic extension, $(\langle B_e \mid e \in \mathbb{E}\rangle, \supset^*\rangle\cong \mathbb E$.
\end{corollary}
\begin{proof}
    Let $e_0, e_1 \in \mathbb{E}$. If $e_0 <_E e_1$ then by (\ref{condition}), and properties of $\mathbb{M}(\dot{F}_{e_1})$, $B_{e_0} \supset^* B_{e_1}$. Otherwise, by Lemma \ref{lemma:incompatibility}, $B_{e_1} \not\subseteq^* B_{e_0}$.
\end{proof}

Thus, the main task which is left is to show how the names $\dot{F}_d$ can be chosen so that the requirements of the corollary hold for a prescribed $\mathbb{E}$ and simultaneously generate an ultrafilter. In the next section we will formally define the iteration by stating precisely how the filters $\dot{F}_d$ are chosen. 



\section{Mathias iteration}\label{Section: Mathias iteration}
The goal of this section is to define for a given well-founded, $\kappa^+$-directed partial order $\mathbb{D}$ a forcing notion that introduces an ultrafilter with a basis isomorphic to $\mathbb{D}$. We will be able to show that this forcing works for the cases $\kappa = \omega$ and $\kappa$ is indestructible supercompact.

Let $\mathbb{D}$ be $\kappa^+$-directed well-founded partial order. Find an ordinal $\mu$ with $|\mu|=|\mathbb{D}|$ and a linearization $\l d_\alpha\mid \alpha<\mu\r$ of $\mathbb{D}$. Namely, $\l d_\alpha\mid \alpha<\mu\r$ enumerates $\mathbb{D}$ and
\begin{align}\label{linearization condition}
    \tag{$\triangle$}\forall i,j<\mu, \ d_i\leq_{\mathbb{D}}d_j\Rightarrow i\leq j.
\end{align} In this situation we say that \textit{$\mu$ is a possible order type of $\mathbb{D}$}. Note that by the $\kappa^+$-directedness if $\mu$ is a possible order type of $\mathbb{D}$, $\cf(\mu)>\kappa$. Let $\mathbb{D}^+=\mu^+\times\mathbb{D}$ with the usual product strict ordering, namely $(\alpha, d) < (\beta, e)$ if and only if $\alpha < \beta$ and $d < e$.  Note that $<$ is not the strict part of the product partial ordering
$(\mu^+, \le) \times (\mathbb{D}, \leq_{\mathbb{D}})$.

Let us assume further that the diamond principle $\diamondsuit_{\mu^+}(\mu^+ \cap \cof(\mu))$
holds in $V$. Here $\cof(\mu)$ denotes the class of ordinals $\alpha$ with $\cf(\alpha)=\cf(\mu)$. 
This form of diamond principle does not follow in general from the assumption that $2^\mu = \mu^+$,
and in the sequel we will arrange it by forcing.
We will define a non-linear iteration $\mathbb{P}=\mathbb{P}_{\mathbb{D}^+}$ of filter Mathias
forcing indexed by $\mathbb{D}^+$, so by Lemma \ref{lemma:basic-properties} 
 $\mathbb{P}$ is
$\kappa$-directed closed and $\kappa^+$-cc.


We will use a form of diamond which is adapted to guessing subsets of $H_{\mu^+}$, which
is possible because $\vert H_{\mu^+} \vert = \mu^+$. We will assume without loss of generality that $\mathbb{D}\in H_{\mu^+}$. It follows from $\diamondsuit_{\mu^+}(\mu^+ \cap \cof(\mu))$
  there are sequences $(A_\alpha)_{\alpha \in \mu^+ \cap \cof(\mu)}$
and $(M_\alpha)_{\alpha < \mu^+}$ such that 
\begin{enumerate}
\item $M_\alpha \prec H_{\mu^+}$ with $\vert M_\alpha \vert =
\mu$ and $M_\alpha \cap \mu^+ \in \mu^+$.
\item $(M_\alpha)_{\alpha < \mu^+}$ is continuous and $\bigcup_{\alpha < \mu^+} M_\alpha = H_{\mu^+}$.
\item $\mathbb{D}\in M_0$ and $\mathbb{D}\cup\mu+1\subseteq M_0$. 
\item $A_\alpha \subseteq M_\alpha$, and for every $A \subseteq H_{\mu^+}$ there are stationarily
  many $\alpha \in \mu^+ \cap \cof(\mu)$ such that $A \cap M_\alpha = A_\alpha$.
\end{enumerate}
The choice of the sequence $(M_\alpha)_{\alpha<\mu^+}$ is not important, since such sequences always exist
and any two such sequences agree on a  club subset of $\mu^+$. Also note that $M_\alpha\cap \mathbb{D}^+=(M_\alpha\cap\mu^+)\times\mathbb{D}$


Before we continue to the definition of $\mathbb{P}$, let us set  up some terminology. Let $\mathbb{Q}$ be a forcing poset and let $\lambda$ be a cardinal. 
It will be convenient to specify exactly what we mean by a {\em canonical $\mathbb{Q}$-name} for a subset of $\lambda$ or a  family of subsets
of $\lambda$. A canonical $\mathbb{Q}$-name for a subset of $\lambda$ is a name $\dot \tau$ such that
$\dot \tau \subseteq \mathbb{Q} \times \{ \check \alpha: \alpha < \lambda \}$, and
$\{ q : (q, \check \alpha) \in \dot \tau \}$ is an antichain for all $\alpha$. 
A canonical name for a family of subsets of $\lambda$ is a name $\dot \sigma$ such that
$\dot \sigma \subseteq \mathbb Q \times \{ \dot \tau : \mbox{$\dot\tau$ is a canonical name for a subset of $\lambda$} \}$,
and $\{ q : (q, \dot \tau) \in \dot \sigma \}$ is an antichain for all $\dot \tau$. 
It is easy to see that any name for a subset of $\lambda$ or a family of subsets of $\lambda$ is equivalent to a canonical name.

As usual, ${\mathbb P}_{<a}$ is defined by recursion on $a\in\mathbb{D}^+$, with a minor modification that we are going to describe.  
As discussed above, the main part of the definition is choosing a ${\mathbb P}_{<a}$-name $\dot F_a$ for a $\kappa$-complete uniform filter on $\kappa$ in $V^{\mathbb{P}_{<a}}$.
\begin{definition}
     $\dot F_a$ is defined in $V$ by induction on $\beta < \mu^+$, and an subsidiary induction on $a \in \{\beta\}\times\mathbb{D}$.
Assume that $\mathbb{P}_{\beta\times \mathbb{D}}$ was defined. 

If $\beta \in \mu^+ \cap \cof(\mu)$
and $A_\beta$ is a canonical ${\mathbb P}_{\beta\times\mathbb{D}}$-name for a uniform $\kappa$-complete ultrafilter on $\kappa$ (abbreviated by ``$\beta$ is an {\em active stage}"), we set for every $a=(\beta,d) \in \{\beta\}\times\mathbb{D}$, \begin{align*}
\dot{F}_a:=\{(1,\dot A)\colon \dot{A}\text{ is a canonical
 }&{\mathbb P}_{\beta\times\mathbb{D}_{<d}}\text{-name} \\
 &\text{ for a subset of $\kappa$ and }\Vdash^V_{ {\mathbb P}_{\beta\times \mathbb{D}} } \dot A \in A_\beta\}.    
\end{align*}
If $\beta$ is not an active stage let $\dot{F}_a$ name the tail filter on $\kappa$. 
\end{definition}

The forcing expression $\Vdash^V_{ {\mathbb P}_{\beta\times \mathbb{D}} } \dot A \in A_\beta$ is reasonable because $\beta\times\mathbb{D}_{<d}$ is a downwards
closed subset of $\beta\times\mathbb{D}$, so that
$\dot A$ makes sense as a ${\mathbb P}_{\beta\times\mathbb{D}}$-name for a subset of $\kappa$. 
It is easy to verify that in either case $\dot F_a$ is a ${\mathbb P}_{\beta\times\mathbb{D}_{<d}}$-name
for a $\kappa$-complete filter on $\kappa$, where we note that 
${\mathbb P}_{\beta\times\mathbb{D}_{<d}} = {\mathbb P}_{<a}$.

 Recall that for a $V$-generic filter $G\subseteq\mathbb{P}$, for any $a\in \mathbb{D}^+$, $G_{<a}$ is the induced generic for $\mathbb{P}_{<a}$. For $\alpha<\mu^+$ we denote by $G_\alpha$ the induced generic for $\mathbb{P}_{\alpha\times \mathbb{D}}$.
 In the next remark, we record a few easy observations, where we make repeated use of the $\kappa^+$-cc and the regularity of $\mu^+$ to
analyze how subsets of $\kappa$ will appear in the generic extension.
\begin{remark}\label{remark: directedness} \ {}
\begin{enumerate}
\item [(I)] Every canonical $\mathbb P$-name for a subset of $\kappa$ is a canonical ${\mathbb P}_a$-name for some $a \in \mathbb{D}^+$.
  It follows that if $G$ is $\mathbb P$-generic then
  $$P(\kappa) \cap V[G] = \bigcup_{a \in \mathbb{D}^+} P(\kappa) \cap  V[G_{<a}]$$ where $G_{<a}$ is the
  projection of $G$ to ${\mathbb P}_{<a}$.
\item [(II)] Similarly,
\begin{enumerate}
    \item [(IIa)] if  $\beta < \mu^+$ and $\cf(\beta) > \kappa$, then
  \begin{align*}
  P(\kappa) \cap V[G_{\beta}] = \bigcup_{\alpha<\beta} P(\kappa) \cap
  V[G_{\alpha}].  \end{align*}
  \item [(IIb)] for any $\beta<\mu^+$, \begin{align*}P(\kappa)\cap V[G_\beta]=\bigcup_{d\in \mathbb{D}}P(\kappa)\cap V[G_{<(\beta,d)}]\end{align*}
 In particular, if $\beta$ is an active stage, then
 \begin{align*}
A_\beta[G_\beta]=\bigcup_{d\in\mathbb{D}}\dot{F}_{(\beta,d)}[G_{<(\beta,d)}].
 \end{align*}
   
   \end{enumerate}
\end{enumerate} 
\end{remark}
Let $\dot U$ be a canonical $\mathbb{P}$-name for a $\kappa$-complete ultrafilter over $\kappa$, and for
 $\beta < \mu^+$ consider the following $\mathbb{P}_{\beta\times \mathbb{D}}$-name
 \begin{align*}
     \dot U_\beta = \{ (q, \dot \tau) \in \dot{U}:&
 q \in {\mathbb P}_{\beta\times\mathbb{D}}\text{ and}\\
 & \dot \tau\text{ is a canonical
   }{\mathbb P}_{\beta\times\mathbb{D}}\text{-name for a subset of }\kappa\}
 \end{align*}
 Then:
 \begin{itemize}
 \item $\dot U \subseteq H_{\mu^+}$.
 \item $\dot U = \bigcup_{\beta < \mu^+} {\dot U}_\beta$, and
   $\langle\dot U_\beta\mid \beta<\mu^+\rangle$ is continuous at points of cofinality greater than $\kappa$.
 \item For all but non-stationary many $\beta \in \mu^+ \cap \cof(\mu)$,
   $\dot U_\beta=\dot{U}\cap M_\beta$ is a name for a $\kappa$-complete ultraiflter on $\kappa$,
   and is forced to equal $\dot{U} \cap V[\dot{G}_{\beta\times\mathbb{D}}]$. Denote the set of such $\beta$ by $S_{\dot{U}}$.

 \end{itemize}
Consider the following crucial set 
$$S^*_{\dot{U}} := \{ \beta \in S_{\dot{U}} : A_\beta = \dot U_\beta \},$$ then by the diamond assumption,
   $S^*_{\dot{U}}$ is a stationary  subset of $\mu^+ \cap \cof(\mu)$, and clearly every $\beta \in S^*_{\dot{U}}$ is an ``active stage'' in the
   definition of $\mathbb P$. Note that for every $\beta\in S^*_{\dot{U}}$:
   \begin{align}\label{equation: S points}
   \dot{U}[G]\cap V[G_{\beta}]=\dot{U}_\beta[G_{\beta}]=A_\beta[G_{\beta}]\in V[G_{\beta}].\tag{$\star$}
   \end{align}

\begin{claim}\label{claim: union of F}
     Let $\dot{U}$ be a canonical $\mathbb{P}$-name for a $\kappa$-complete ultrafilter over $\kappa$, let $\alpha\in S^*_{\dot{U}}$, and $G_\alpha\subseteq\mathbb{P}_{\alpha\times\mathbb{D}}$ be $V$-generic. Then:
     \begin{enumerate}
         \item For all $\beta\in S^*_{\dot{U}}\cap\alpha$, $\dot{U}_\alpha[G_\alpha]\cap V[G_\beta]=\dot{U}_\beta[G_\beta]$
     \item $\dot{U}_\alpha[G_\alpha]=\bigcup_{d\in\mathbb{D}}\dot{F}_{(\alpha,d)}[G_{<(\alpha,d)}]$.
     \end{enumerate}
\end{claim}
\begin{proof}
$(1)$ follows straightforwardly from (\ref{equation: S points}). For $(2)$,
    by the definition of $\dot{F}_{\alpha,d}$, and since $\dot{U}_\alpha=A_\alpha$, the inclusion from right to left is immediate. Let $x  = {\dot x}[G_\alpha] \in \dot{U}_\alpha[G_\alpha]$, let $p \in G_\alpha$ force
    $\dot x \in {\dot U}_\alpha$. Using the maximum principle and the fact that any name is forced
    to be equal to a canonical name, we find a canonical ${\mathbb P}_{\alpha \times \mathbb{D}}$
    name $\dot y$ such that for any generic $H$, 
    ${\dot y}[H] = {\dot x}[H]$ when $p \in H$ and ${\dot y}[H] = \kappa$ when $p \notin H$. 
    Then $x = {\dot y}[G_\alpha]$ and 
     $\Vdash_{\mathbb{P}_{\alpha\times\mathbb{D}}}\dot{y}\in \dot{U}_\alpha$. By the $\kappa^+$-cc and the $\kappa^+$-directedness of $\mathbb{D}$, we can find $d\in \mathbb{D}$ such that  $\dot{y}$ is a $\mathbb{P}_{<(\alpha,d)}$-name. It follows that $x\in \dot{F}_{(\alpha,d)}[G_{<(\alpha,d)}]$.
\end{proof}
The following lemma gives a sufficient condition for the main theorem: 
\begin{lemma}\label{Lemma: aux}
    Let $\dot{U}$ be a canonical $\mathbb{P}$-name for a $\kappa$-complete ultrafilter over $\kappa$ and let $I$ be a subset of $\mu^+$  such that $\{B_{(\gamma,d)} \mid \gamma\in I, d\in\mathbb{D}\}$ generates $\dot{U}[G]$ for some $V$-generic $G\subseteq \mathbb{P}$. Then in $V$ there is a club $C\subseteq \mu^+$, such that for every $\beta\in S^*_{\dot{U}}\cap C$: 
    \begin{enumerate}
        \item  $\{B_{(\gamma,d)} \mid \gamma\in I\cap\beta, d\in\mathbb{D}\}$ generates $\dot{U}_\beta[G_\beta]$.
        \item If $I\times \mathbb{D}$ satisfies (\ref{condition}) of Corollary \ref{cor:diagonalizing-the-ultrafilter}, and there is an unbounded $I^*\subseteq I\cap \beta$ with $\text{otp}(I^*)=\mu$, then the ultrafilter $\dot{U}_\beta[G_\beta]$ is a simple $\mathbb{P}_{\mathbb{D}}$-point.
    \end{enumerate} 
\end{lemma}

\begin{proof}
    For each $\alpha<\mu^+$, let $\beta_\alpha<\mu^+$ be the least such that for every $x\in \dot{U}[G]\cap V[G_\alpha]$ there is $\gamma\in \beta_\alpha\cap I$ and $d\in\mathbb{D}$ such that $B_{(\gamma,d)}\subseteq^* x$. Such a $\beta_\alpha$ exists since in $V[G_\alpha]$, $2^\kappa\leq \mu$. Let $C\subseteq \mu^+$ be the club of closure points of $\alpha\mapsto \beta_\alpha$. By the $\kappa^+$-cc of the forcing, we may assume that $C\in V$. To see $(1)$, let $\beta\in S^*_{\dot{U}}\cap C$,  $x\in \dot{U}_\beta[G_{\beta}]$. By Remark \ref{remark: directedness}(II), there is $\alpha \in S^*_{\dot{U}}\cap\beta$  such that $x\in \dot{U}_\alpha[G_\alpha]$. Since $\beta\in C$, $\beta_\alpha<\beta$ and therefore there is $\gamma\in I\cap \beta$ and $d\in \mathbb{D}$ such that $B_{(\gamma,d)}\subseteq^* x$. 

For $(2)$, recall that $\langle d_\alpha \mid \alpha < \mu\rangle$ is a linearization of $\mathbb{D}$ and
        let $\langle \epsilon_\alpha \mid \alpha < \mu\rangle$ be the increasing enumeration of the members of $I^*$. By Corollary \ref{cor:diagonalizing-the-ultrafilter}, for every $\alpha,\alpha' < \mu$, and $a,b\in \mathbb{D}$, $B_{\epsilon_\alpha, a} \subset^* B_{\epsilon_{\alpha'}, b}$ if and only if ${\alpha'} <\alpha$ and $b <_{\mathbb{D}} a$. Let 
        \[\mathcal{B} = \{B_{(\epsilon_\alpha, d_\alpha)} \mid \alpha < \mu\}.\]
        Then $(\mathcal{B},\supset^*)\simeq (\{(\epsilon_\alpha,d_\alpha)\mid \alpha<\mu\},<_{\mathbb{D}^+})\simeq\mathbb{D}$, where the right isomorphism is given by $\varphi(\epsilon,d)=d$. While this is clearly a bijection, let us varify that is is order preserving: let $\alpha,{\alpha'}<\mu$. Then \begin{align*}
            (\epsilon_{\alpha'},d_{\alpha'})<(\epsilon_\alpha,d_\alpha)&\text{ iff }(\epsilon_{\alpha'}<\epsilon_\alpha\text{ and }d_{\alpha'}<_{\mathbb{D}}d_\alpha) \\
            &\text{ iff }({\alpha'}<\alpha \text{ and }d_{\alpha'}<_{\mathbb{D}}d_\alpha) &
             \epsilon_i\text{'s are increasing}\\
            &\text{ iff }d_{\alpha'}<_{\mathbb{D}}d_\alpha &\text{by condition (\ref{linearization condition}) }\\
            &\text{ iff } \varphi(\epsilon_{\alpha'},d_{\alpha'})<\varphi(\epsilon_\alpha,d_\alpha)
        \end{align*}
        
        To see $\mathcal{B}$ is a base for
        $\dot{U}_\beta[G_\beta]$, it is enough to verify that the elements of the  base described in $(1)$ were generated by $\mathcal{B}$. 
        Let $(\gamma, d)\in (I\cap\beta)\times\mathbb{D}$ be arbitrary. Find $\alpha$ such that $\epsilon_\alpha>\gamma$ (which exists since $I^*$ is unbounded) and let $\alpha'$ be such that $d=d_{\alpha'}$. Let $\alpha^* > \alpha, \alpha'$ such that $d <_{\mathbb{D}} d_{\alpha^*}$. Hence, $B_{(\epsilon_{\alpha^*}, d_{\alpha^*})}$ is contained in $B_{\gamma, d}$. 
    \end{proof}
\begin{remark}\label{remark: additional}
    Note that if we replace the assumption in Lemma \ref{Lemma: aux} with $I\subseteq S^*_{\dot{U}}$ and (\ref{condition}), then Lemma \ref{Lemma: aux} holds for $\beta\in S^*_{\dot{U}}\cap \lim(I)$. Indeed, the argument for $(2)$ is identical. To see $(1)$, first argue that $\{B_{(\gamma,d)} \mid \gamma\in I\cap \beta, d\in\mathbb{D}\}\subseteq\dot{U}_\beta[G_\beta]$: using (\ref{condition}), if $\gamma\in I\cap \beta$, then for all $d\in\mathbb{D}$, $B_{(\gamma,d)}\in \dot{F}_{\beta,d'}[G_{<(\beta,d')}]$ for any $d'>d$. Since $\beta\in S^*_{\dot{U}}$, Claim \ref{claim: union of F} ensures that $\dot{F}_{\beta,d'}[G_{<(\beta,d')}]\subseteq\dot{U}_\beta[G_{\beta}]$.  Conversely, if $x\in \dot{U}_\beta[G_\beta]$ we use Claim \ref{claim: union of F} again to see that $x\in \dot{F}_{(\alpha,d)}[G_{<(\alpha,d)}]$ for some $\alpha\in I\cap \beta$ and $d\in\mathbb{D}$. Thus  $B_{(\alpha, d)} \subseteq^* x$.

\end{remark}
  The point  of Lemma \ref{Lemma: aux} is that in order to finish the argument, it suffices to find an ultrafilter $U$ in the generic extension such that there are generic Mathias sets $B_a$ belonging to $U$ for sufficiently many $a\in\mathbb{D}^+$ such that condition (\ref{condition}) is satisfied.

 The ultimate goal will be to have enough copies of $\mathbb{D}$ (i.e. fibers of $\mathbb{D}^+$) so that cofinally we can get simultaneously a copy of $\mathbb{D}$ which is cofinal. In order to obtain it, we will use a different argument for the countable case and the supercompact case.
\subsection{At a measurable cardinal}\label{Section at a measur}
Here we will prove the main consistency result for measurable cardinals.
\begin{theorem}\label{Cor: Simple P-point measurable}
    Let $\kappa$ be supercompact cardinal. Then for every well-founded, $\kappa^+$-directed set $\mathbb{D}$ with a possible order-type $\mu$, there is a \begin{align*} 
    \text{(}\kappa\text{-cc) }\ast\text{ (}\mu^+\text{-directed closed) }\ast\text{ (}\kappa\text{-directed closed+}\kappa^+\text{-cc)} \end{align*}forcing extension in which there is a simple $P_{\mathbb{D}}$ ultrafilter over $\kappa$.
\end{theorem}
To illustrate the main idea in a simple setting, suppose that $\kappa$ is a
Laver indestructible supercompact cardinal and consider a linear iteration
of ultrafilter Mathias forcing at $\kappa$ with ${<}\kappa$-supports,
where at stage $\alpha$ we use a measure $U_\alpha \in V[G_\alpha]$ to define
the next step in the iteration: there is no problem in choosing
such $U_\alpha$ because the forcing so far is $\kappa$-directed closed.
In the final model $V[G]$ there is of course a measure $U$, but
{\em a priori} we have no reason to believe that the set $B_\alpha$ added at stage
$\alpha$ is in $U$, or even that $U_\alpha \subseteq U$.

This problem was first solved by D{\v z}amonja and Shelah \cite{DzSh},
where the rough idea is that we force at $\alpha$ with the disjoint 
sum of all measures which exist in $V[G_\alpha]$. This leads to some
technical complications in the  iteration, which we can avoid
using the diamond machinery we used to define the iteration which traces back to Magidor. The key feature of the construction of $U$ is to ensure that whenever $U_\alpha \subseteq U$ (which happens frequently)
we have $B_{(\alpha,d)} \in U$ for every $d\in\mathbb{D}$. Running the iteration for a large enough number of steps, we may  produce
 measures which are generated by arbitrarily long $\subseteq^*$-decreasing sequences.  
The diamond idea appears in a paper by Cummings, D{\v z}amonja, Magidor,  Morgan and Shelah \cite{5authors}
albeit in a more complex setting.


Moving on towards the proof of \ref{Cor: Simple P-point measurable}, 
we start with a model $V_0$ in which $\kappa$ is supercompact
and $\mathbb{D}$ is a suitable directed set with a linearisation $\langle d_\alpha: \alpha < \mu \rangle$.
We start by forcing with
the Laver preparation $\mathbb{L}$ from \cite{Laver1978} which makes $\kappa$ indestructibly supercompact.  Next, we let $\mathbb{Q}$ be the standard forcing to add a $\diamondsuit(\mu^+\cap cf(\mu))$-sequence by initial segments. We note that this forcing is $\mu^+$-directed closed and if $2^\mu>\mu^+$, it collapses $2^\mu$ to $\mu^+$. We obtain a model\begin{align*}
    V \models & ``\kappa \mbox{ is supercompact} 
     \mbox{ indestructible under }\\ &\kappa\mbox{-directed closed forcing}+\diamondsuit(\mu^+\cap\cof(\mu))".
\end{align*}
Working in $V$ we let $\mathbb P = {\mathbb P}_{{\mathbb D}^+}$.

We will prove:
\begin{theorem} \label{msble} 
  If $G$ is $\mathbb P$-generic then there exist in $V[G]$ a normal measure $U$ on
  $\kappa$, such that a canonical name $\dot{U}$ for $U$ will have  the set $S^*_{\dot{U}}\times \mathbb{D}$ satisfy (\ref{condition}).
\end{theorem} 
\begin{proof}[Proof of Theorem \ref{msble}]
Working in $V_0$ we fix $j: V_0 \rightarrow M_0$ such that $j$ witnesses that $\kappa$ is $\mu^+$-supercompact 
and $j(\mathbb L) \simeq {\mathbb L} * {\mathbb P} * {\mathbb R}$ where
 the first point in the support of the tail iteration $\mathbb R$ is greater than $\mu^+$. 
 This ensures that:
 \begin{enumerate}
 \item $\Vdash^{M_0}_{ {\mathbb L} * {\mathbb P} } \mbox{``$\mathbb R * j(\mathbb P)$ is $\mu^{++}$-closed''}$.
 \item  $V[G] \models {}^{\mu^+} M_0[L][G] \subseteq M_0[L][G]$
 \item $V[G] \models \mbox{``$\mathbb R * j(\mathbb P)$ is $\mu^{++}$-closed''}$.
 \end{enumerate}

 It is straightforward
 to lift $j: V_0 \rightarrow M_0$ to obtain
 $j: V = V_0[L] \rightarrow M_0[L * G * H]$ for any $H$ which is $\mathbb R$-generic  
 over $V[G]$. To further lift $j$ onto $V[G]$ 
 we would need a $j({\mathbb P})$-generic object $S$ such that
 $j[G] \subseteq S$. It is worth noting that, since conditions
 in ${\mathbb P}$ have supports of size less than $\kappa$,
 the support of a condition in $j[G]$ will be contained in
 $j[\mathbb{D}^+]$.

 Let $\prec$  be the ``lexicographic" ordering of $\mathbb{D}^+$
 given by $(\epsilon, d_\alpha) \prec (\zeta, d_\beta)$
 if and only if eithe r $\epsilon < \zeta$ or
 $\epsilon = \zeta$ and $\alpha < \beta$. 
 It is easy to see that:
 \begin{itemize}
 \item $\prec$ is  a  well-ordering of $\mathbb{D}^+$ with order type $\mu^+$.
 \item $\prec$ extends $<_{ {\mathbb D}^+}$. 
 \end{itemize}
 Denote by $\{e_\zeta\mid \zeta<\mu^+\}$ the enumeration of $\mathbb{D}^+$ in $\prec$-increasing order.

 Working in $V[G]$, we will construct a decreasing sequence
 $(r_i, s_i)_{i < \mu^+}$ of conditions in ${\mathbb R} * j(\mathbb P)$.
 We will arrange that:
 \begin{itemize}
 \item The support of $s_i$ is contained in
   $j( \{ e_\eta : \eta < i \})$, so that in particular
   $j(e_\zeta)$ is not in the support of $s_i$ for $\zeta \ge i$. 
 \item $s_i \le j[G_{<e_i}]$, so that 
   $(r_i, s_i)$ forces that $j$ can be lifted onto  $V[G_{<e_i}]$.
 \end{itemize}
 We note that $s_i$ may not actually be a condition in $j({\mathbb P}_{<e_i})$,
 the point is that the projection of $s_i$ to $j({\mathbb P}_{<e_i})$
 forces that the $j({\mathbb P}_{<e_i})$-generic object contains $j[G_{<e_i}]$.  
 
 For $i = 0$, there is no element of $\mathbb{D}^+$ below $e_0$, 
 so let $(r_0, s_0)$ be the trivial condition.
 For $i$ non zero limit ordinal let $(r_i, s_i)$ be a lower bound for
 $(r_j, s_j)_{j < i}$, taking care that
 the support of $s_i$ is the union of the supports of
 $s_j$ for $j < i$. 

 To construct $(r_{i+1}, s_{i+1})$, we force below
 $(r_i, s_i)$ and construct
 $j: V[G_{<e_i}] \rightarrow M_0[L][G][H][S']$ where
 $H * S'$ is ${\mathbb R} * j({\mathbb P}_{<e_i})$-generic.  
 Recall that we defined ${\mathbb Q}_{e_i}$ from a
 $\kappa$-complete uniform filter $F_{e_i} \in V[G_{<e_i}]$,
 and that the $e_i$-component of $G$ gives
 a subset $B_{e_i}$ of $\kappa$ which is  ${\mathbb Q}_{e_i}$-generic. 

 Let $s'_i$ be the projection (i.e.\ restriction) of $s_i$ to $j({\mathbb P}_{<e_i})$.
 There are now two cases:
 \begin{enumerate}
   
 \item \label{case1} $(r_i, s'_i) \Vdash^{M_0[L][G]}_{ {\mathbb R} * j( {\mathbb P}_{<e_i} ) }  \kappa \in j(E)$ for every $E \in F_{e_{i}}$.
   In this case we define $r_{i+1}$ and $s_{i+1}$ as follows:
   \begin{enumerate}
   \item $r_{i+1} = r_i$.   
   \item $\dom(s_{i+1}) = \dom(s_i) \cup \{ j(e_i) \}$.
   \item $s_{i+1} \restriction \dom(s_i) = s_i$.
   \item $s_{i+1}(j(e_i)) = (B_{e_i} \cup \{ \kappa \}, \bigcap j[F_{e_i}])$.
   \end{enumerate}

 \item \label{case2} There is $E \in F_{e_i}$ such that $(r_i, s'_i) \not\Vdash^{M_0[L][G]}_{ {\mathbb R} * j( {\mathbb P}_{<e_i} ) }\kappa \in j(E)$.
   In this case we choose such an $E_i=E$ and find $(r_i^*, s_i^*) \in {\mathbb R} * j({\mathbb P}_{<e_i})$
   such that $(r_i^*, s_i^*) \le (r_i, s'_i)$ and
   $(r_i^*, s_i^*) \Vdash \kappa \notin j(E_i)$. We then define:
   \begin{enumerate}
   \item $r_{i+1} = r_i^*$.   
   \item $\dom(s_{i+1}) = \dom(s^*_i) \cup (\dom(s_i) \setminus \dom(s^*_i)) \cup \{ j(e_i) \}$.
   \item $s_{i+1} \restriction \dom(s^*_i) = s^*_i$.
   \item $s_{i+1} \restriction (\dom(s_i) \setminus \dom(s^*_i)) = s_i \restriction (\dom(s_i) \setminus \dom(s^*_i))$.
   \item $s_{i+1}(j(e_i)) = (B_{e_i}, \bigcap j[F_{e_i}])$.
   \end{enumerate}

 \end{enumerate}

 Let us check that we have maintained the hypotheses. The main points are:
 \begin{enumerate}
 \item $j(F_{e_i})$ is $j(\kappa)$-complete and $j[F_{e_i}] \in M_0[L][G][H][S']$ so that $\bigcap j[F_{e_i}] \in j(F_{e_i})$ and
   we constructed legitimate conditions in $j({\mathbb Q}_{e_i})$.
 \item In either case of the construction,
   $\dom(s_{i+1}) \subseteq \dom(s_i) \cup j({\mathbb{D}_{<e_i}}) \cup \{ j(e_i) \}$.
   Since $\mathbb{D}_{<e_i} \subseteq \{ e_\eta: \eta < i \}$, it follows that
   $\dom(s_{i+1}) \subseteq j( \{ e_\zeta: \zeta \le i \})$. 
 \item The filter added by $G$ at $e_i$ is the set of pairs $(s, E)$ where $E \in F_{e_i}$,
   $s$ is an initial segment of $B_{e_i}$ and $B_{e_i} \setminus s \subseteq E$.
   Since $j ( (s, E) ) = (s, j(E))$, it is easy to see that
   it is forced that $(B_{e_i}, \bigcap j[F_{e_i}]) \le (s, j(E))$. In the first case
   $(r_i, s_i)$ forces that $(B_{e_i} \cup \{ \kappa \}, \bigcap j[F_{e_i}]) \le (s, j(E))$.
 \end{enumerate}

 When the construction of $(r_i, s_i)_{i < \mu^+}$ is done, we choose $(r', s')$ to be a lower bound
 and note that $(r', s')$ forces that $j$ can be lifted onto $V[G]$.
 Recall that every subset of $\kappa$ in $V[G]$ lies in $V[G_{<a}]$ for some $a \in \mathbb{D}^+$, and
 observe that if $$X =  \{ \dot A : \text{For some }a \in \mathbb{D}^+,\ \dot A\text{ is a canonical }{\mathbb P}_{<a}\text{-name for a
   subset of }\kappa \},$$
 then $\vert X \vert =  \mu^+$. Using the closure of ${\mathbb R} * j(\mathbb P)$ again,
 we find $(r, s) \le (r', s')$ such that $(r, s)$ decides $\kappa \in j(\dot A)$ for every
 $\dot A \in X$. It is now easy to see that
 if $U = \{ {\dot A}[G] :\mbox{$\dot A \in X$ and $(r, s) \Vdash \kappa \in j(\dot A)$} \}$
 then $V[G] \models \mbox{``$U$ is a measure on $\kappa$''}$.

 Let $\dot U$ be a canonical name for the measure $U$ produced by this construction and let us verify (\ref{condition}) of Corollary \ref{cor:diagonalizing-the-ultrafilter}. This is done in Claim \ref{keyclaim2}:
\begin{claim} \label{keyclaim} 
   For every $a \in S^*_{\dot{U}}\times \mathbb{D}$,  $B_a \in U$.
\end{claim}
\begin{proof} 
  Let $a = e_i = (\beta, d)$ say,
  and consider the construction of $(r_{i+1}, s_{i+1})$.  The key point is that we must be in case \ref{case1}
  of the construction, so assume for a contradiction that we are in case \ref{case2}.
  Since $(r, s) \le (r_{i+1}, s_{i+1})$, this means that there is $E \in F_{e_i}$ such that
  $E \notin U$. 
    Since $\beta$ is an active stage, $A_\beta = \dot U_\beta$, and thus
    $E \in U \cap V[G_ {\beta\times\mathbb{D}}]$, which is a contradiction.

    Since we are in case \ref{case1}, $(r_{i+1}, s_{i + 1}) \Vdash \kappa \in j(B_a)$. So
    $(r, s) \Vdash \kappa \in j(B_a)$, and therefore $B_a \in U$. 
\end{proof}
As a corollary, we conclude that whenever $\beta$ is an active stage of the iteration the sets $B_a$ for $a\in \{\beta\}\times\mathbb{D}$ belong to  $U$, and therefore generate a $\kappa$-complete filter. 

\begin{claim} \label{keyclaim2} Let $a, a' \in S^*_{\dot{U}}\times \mathbb{D}$ where $a < a'$.
  Then $\Vdash_{\mathbb{P}_{<a'}}\dot{B}_a \in \dot{ F}_{a'}$. 
\end{claim} 
\begin{proof}
 Let $a=(\beta,d)$
  and $a'=(\beta',d')$. By the definition of the order, $\beta < \beta'$. 
  Then, clearly
  $B_a \in V[G_{\beta'\times\mathbb{D}_{<a'}}]$.  Let $\dot B_a$ be a canonical ${\mathbb P}_{<a'}$-name for
  $B_a$. We need to show that
  $\Vdash^V_{ {\mathbb P}_{\beta'\times\mathbb{D}} } \dot B_a  \in A_{\beta'}$.
  Let $G_{\beta'\times\mathbb{D}}$ be an arbitrary generic
  for ${\mathbb P}_{\beta'\times\mathbb{D}}$, and force to
  prolong it to a $\mathbb{P}$-generic $G$. Then by Claim \ref{keyclaim} and Equation (\ref{equation: S points})
  \[
  B_a \in U \cap V[G_{\beta'\times\mathbb{D}}]
  = \dot{U}_{\beta'}[G_{\beta'\times\mathbb{D}}]
  = A_{\beta'}[G_{\beta'\times\mathbb{D}}]
  \]
  Note that this is forced by the trivial condition.
  Hence, by definition, $(1,\dot{B}_a)\in \dot{F}_{a'}$ and $B_a\in F_{a'}$. 
\end{proof}

This concludes the proof of Theorem \ref{msble}.
  \end{proof}
We are ready to prove Theorem \ref{Cor: Simple P-point measurable}. Since (\ref{condition}) was already proven, we would like to apply Lemma \ref{Lemma: aux} to some $I\subseteq S^*_{\dot{U}}$:
\begin{proof}[Proof of Theorem~\ref{Cor: Simple P-point measurable}] Let $\delta \in S^*_{\dot{U}}$ be a limit point of $S^*_{\dot{U}}$ such that there is $I\subseteq S^*_{\dot{U}}$ unbounded in $\delta$ with $otp(I)=\mu$. To find such an $\delta$, simply note that the set of all limit points of $S^*_{\dot{U}}$ which have such a set $I$ is a club relative to $\cf(\mu)$. Applying Lemma \ref{Lemma: aux} or more precisely Remark \ref{remark: additional}, we conclude that $U\cap V[G_{\delta\times\mathbb{D}}]=\dot{U}_\delta[G_\delta]$ is a simple $\mathbb{P}_{\mathbb{D}}$-point.
\end{proof}

\begin{corollary}\label{Cor: Clubs}
    From the model of the previous theorem, for every $\kappa^{+}$-directed well-founded poset  $\mathbb{D}$,  there is a $\kappa^+$-cc and $\kappa$-directed closed forcing that adds a base $\mathcal{C}$ for the club filter $\text{Cub}_\kappa$ such that $(\mathcal{C},\supseteq^*)\simeq \mathbb{D}$.
\end{corollary}
\begin{proof}
Let $\mathcal{B}$ be the base from the previous theorem and let $\mathcal{C}$ be the set of closures of the sets of $\mathcal{B}$. Then since $\mathcal{B}$ generated a normal ultrafilter, $\mathcal{C}$ generates $\text{Cub}_\kappa$. Clearly, if $X\subseteq^* Y$ then $cl(X)\subseteq^* cl(Y)$. A closer look at the proof of Corollary \ref{cor:diagonalizing-the-ultrafilter} will reveal that if $X,Y\in \mathcal{B}$ and $X\not\subseteq^* Y$, then $X\not\subseteq^* cl(Y)$. Indeed, in Lemma \ref{lemma:incompatibility}, the choice of $\gamma$ is such that $\gamma$ is forced to be outside of the closure of $\dot{B}_{e}$. Therefore, $cl(X)\not\subseteq^* cl(Y)$.
\end{proof}
\subsection{The countable case}\label{section: omega case}
\begin{theorem}\label{thm: the omega simple p-point}
    Let $\mathbb{D}$ be a well-founded, $\sigma$-directed set, which has a possible order type  $\lambda$, and assume that 
$\diamondsuit_{\lambda^{+}}(\lambda^{+}\cap\cof(\lambda))$ holds. 
    Then, there is a ccc \ forcing extension in which there is a simple $P_\mathbb{D}$-point ultrafilter on $\omega$. 
\end{theorem}
\begin{proof}
    Let $\mathbb{P}=\mathbb{P}_{\mathbb{D}^+}$ be the Mathias iteration from section \S\ref{Section: Mathias iteration}.
    
    Let us work in a generic extension $V[G]$. For each $a\in\mathbb{D}^+$,  recall that $B_a$ denotes the Mathias real added by this coordinate.  
    
    Let $I \subseteq \lambda^{+}$ with the following properties:
    \begin{itemize}
        \item [$(A)$] Every $\alpha \in I$ is active. 
        \item [$(B)$] For every $\alpha < \beta$ in $I$, for every $d\in \mathbb{D}$, \[\Vdash_{\mathbb{P_{\beta\times\mathbb{D}}}} \dot{B}_{\alpha, d} \in A_\beta.\]
        (Recall that $A_\beta$ denotes the $\beta^{\text{th}}$-element in the diamond sequence)
        \item [$(C)$] $I$ is maximal among sets with  properties $(A),(B)$. 
    \end{itemize}

    By Zorn's lemma, $I$ exists, and it is defined in $V$. 
    We will show that $I$ must be unbounded, and that  $\{B_{(\alpha,d)}$: $\alpha\in I, d\in \mathbb{D}\}$ generates an ultrafilter.

    \begin{claim}\label{claim:adding-one-row} Let $G_\alpha\subseteq\mathbb{P}_{\alpha\times\mathbb{D}}$ be $V$-generic for some active level $\alpha$. Then the collection  $\{B_{(\alpha,d)}$: $d\in \mathbb{D}\}$ generates a proper filter which extends $ A_\alpha[G_\alpha]$. 
    \end{claim}
    \begin{proof}
        Let us show that the collection $\{B_{(\alpha,d)} \mid d\in \mathbb{D}\}$ has the finite intersection property. Assume otherwise, then working in $V[G_\alpha]$, there is a condition $p\in \mathbb{P}_{(\alpha+1)\times\mathbb{D}}$
        with $p \restriction \alpha \times {\mathbb D} \in G_\alpha$ forcing that the intersection of some finite subset of $\{\dot{B}_{(\alpha,d)} \mid d\in \mathbb{D}\}$  is finite. By strengthening $p$ we may find a specific finite set $b\subseteq \mathbb{D}$ and $l<\omega$ such that $p$ forces $\bigcap_{x\in b} \dot{B}_{(\alpha,x)} \subseteq l$. 

        Let us look at $p(\alpha, x) = \langle \check t_x, \dot{S}_x\rangle$. 
        As we are working in $V[G_\alpha]$ we may treat $S_{x} = \dot{S}_{x}[G_{\alpha}]$ as actual sets and not merely names. 
        As $S_{x} \in A_\alpha[G_\alpha]$ for each $x$, $ \bigcap_{x\in b}  S_{x}$ is infinite and in particular there is some $l\leq k < \omega$ in this intersection which is larger than $\max t_x$ for all $x\in b$. Extend the condition $p\restriction \alpha\times\mathbb{D}$ to $p^*\restriction \alpha\times\mathbb{D}\in G_\alpha$ so that $p^*\restriction \alpha\times\mathbb{D}\Vdash k\in \bigcap_{x\in b}\dot{S}_x$, and extend each $p(\alpha,x)$ to $p^*(\alpha,x)$ by adding $k$ to $t_x$ for all $x\in b$. $p^*$ is a legitimate condition extending $p$, which forces contradictory information. 

        The fact that the filter generated by $\{B_{(\alpha,d)}\mid d\in\mathbb{D}\}$ extends $A_\alpha[G_\alpha]$ follows from the observation in Remark \ref{remark: directedness}(IIb).
    \end{proof}
    \begin{corollary}
    The collection $\{B_{(\alpha, a)} \mid \alpha \in I, a\in \mathbb{D}\}$ has the finite intersection property.
    \end{corollary}
    \begin{proof}
   We use Claim \ref{claim:adding-one-row}: Indeed, this would be trivial if $I$ happens to be empty, so suppose that $I\neq\emptyset$. 
    If $I$ has a maximal element $\alpha_*$, and we force to obtain $G_{\alpha^*}$ which is ${\mathbb P}_{\alpha^* \times {\mathbb D}}$-generic,
    then $\{B_{(\alpha_*, a)} \mid a\in \mathbb{D}\} \cup A_{\alpha^*}[G_{\alpha^*}]$ is a collection that contains $\{B_{(\beta, a)} \mid \beta \in I, a\in \mathbb{D}\}$ by property $(B)$, and by property $(A)$ we may apply Claim \ref{claim:adding-one-row} to conclude that this collection has the finite intersection property. 

    Otherwise, for any given finite collection of elements in $I\times \mathbb{D}$, we can chop $I$ up to a sufficiently large ordinal $\alpha^*\in I$ and reduce the problem to the previous case.\end{proof}
    Let $\mathcal{F}$ be the filter generated by $\{B_{(\alpha,d)}\mid \alpha \in I, d\in \mathbb{D}\}$ and let $\mathcal{U}$ be an ultrafilter extending it in $V[G]$. Let $\dot{U}$ be a nice name for the ultrafilter $\mathcal{U}$ with the property that every condition forces that it contains $\{B_{(\alpha, d)} \mid \alpha \in I,\, d\in \mathbb{D}\} $. 
    \begin{claim}\label{claim: I is an ultrafilter}
        $I$ is unbounded, and $\{B_{(\alpha, d)} \mid \alpha \in I,\, d\in \mathbb{D}\}$ generates  $\mathcal{U}$.
    \end{claim}
    \begin{proof}
        For the first claim, assume that $I$ is bounded and let $\beta < \lambda^{+}$ large enough so that $I\subseteq \beta$.    
        Pick $\gamma > \beta$ in which the restriction of $\dot{U}$ to $\mathbb{P}_{\gamma \times \mathbb{D}}$ is guessed by the diamond sequence. Let us denote this name for an ultrafilter by $\dot{U}_\gamma$ (i.e. $A_\gamma=\dot{U}_\gamma$). 
        Clearly, $I\cup \{\gamma\}$ satisfies the properties $(A)-(B)$, contradicting the maximality in property $(C)$. 

        For the second claim,        
        since $\{B_{(\alpha,d)}\mid d\in \mathbb{D}\}$ generates in $V[G_{\alpha+1}]$ a filter 
        extending the $V[G_\alpha]$-ultrafilter $A_\alpha[G_\alpha]$, we conclude that $\{B_{(\alpha,d)}\mid d\in \mathbb{D}\}$ 
        generates a $V[G]$-ultrafilter which then has to be $\mathcal{U}$.
    \end{proof}
    To finish the proof, use Claim \ref{claim: I is an ultrafilter} to see that $I$ satisfies the assumption of Lemma \ref{Lemma: aux} and produce the club $C$. Pick any point $\beta\in C\cap S^*_{\dot{U}}$, such that for some unbounded $I^*\subseteq I\cap \beta$,  $\text{otp}(I^*)=\lambda$. Such a point exists by the stationarity of $S^*_{\dot{U}}\cap C$ and the observation that there is a club relative to $\text{cof}(\lambda)$ of $\beta$'s such that there is $I^*\subseteq I$, $\text{otp}(I^*)=\lambda$ and $I^*$ is unbounded in $\beta$. 
    By the definition of $I$, $I^*\times\mathbb{D}$ satisfies (\ref{condition}) and by the previous claim 
    $\{B_{(\alpha,d)} \mid \alpha \in I^*, d\in \mathbb{D}\}$ generates an ultrafilter, so Lemma \ref{Lemma: aux} applies.
\end{proof}
\section{Applications to cofinal types and cardinal characteristics}\label{Section: Tukey order}
In this section we derive some corollaries from our main result. Let us formulate a simple lemma which explains what happens to the isomorphism or Tukey type of the ordered set $\mathbb{D}$ under those mild forcing extensions. Clearly, changes are made (as we are adding an ultrafilter which is order isomorphic to it), but the essential structure of $\mathbb{D}$ is preserved. Note that in any forcing extension, a ground model unbounded (cofinal) map remains unbounded (cofinal).  
\begin{lemma}
Let $\mathbb{D}$ be a $\kappa^+$-directed set and $\mathbb{P}$ a $\kappa^+$-cc forcing and let $G\subseteq\mathbb{P}$ be $V$-generic. Then for any $f:A\to \mathbb{D}\in V[G]$, $A\in V$, there is $f':A\to\mathbb{D}\in V$ such that $f(a)\leq f'(a)$ for all $a\in A$.
\end{lemma}
\begin{proof}
    This is  straightforward from the $\kappa^+$-cc and $\mathbb{D}$ being $\kappa^+$-directed.
\end{proof}
\begin{corollary}\label{cor: comparison in generic}
    Let $A,\mathbb{D}\in V$ and assume that $\mathbb{D}$ is $\kappa^+$-directed, $\mathbb{P}$ is a $\kappa^+$-cc forcing notion and $G\subseteq\mathbb{P}$ is $V$-generic. Then:
    \begin{enumerate}
        \item  $V[G]\models A\leq_T \mathbb{D}$ iff $V\models A\leq_T \mathbb{D}$.
        \item $V[G]\models A\geq_T \mathbb{D}$ iff $V\models A\geq_T \mathbb{D}$.
    \end{enumerate}
\end{corollary}
\begin{proof}
    In both $(1)$ and $(2)$ we are left to show the implication from left to right. For $(1)$, let $f:A\to \mathbb{D}\in V[G]$ be unbounded and let $f':A\to \mathbb{D}\in V$ be the map from the previous lemma. We claim that $f'$ is unbounded, otherwise, there is $\mathcal{A}\subseteq A$ unbounded and $d\in \mathbb{D}$ such that $f'[\mathcal{A}]$ is bounded by $d$. Clearly, the boundedness of a subset of $\mathbb{D}$ is absolute, and since $f$ is an unbounded map, $f[\mathcal{A}]$ should be unbounded. But for every $a\in A$, $f(a)\leq f'(a)\leq d$, contradiction.
     $(2)$ is proven similarly.
    \end{proof}
The following set is also known as the \textit{point-spectrum}. It was studied for example in \cite{Isbell65,Gilton,TomCohesive,GARTSIDE2021102873}.
\begin{definition}\label{Def: Tukey Spec}
    For a directed set $\mathbb{D}$ we define the \emph{Tukey spectrum} of $\mathbb{D}$, $\spec_T(\mathbb{D})$ to consist of all regular cardinals $\lambda$ such that $\lambda\leq_T \mathbb{D}$.
\end{definition}
Clearly, the Tukey spectrum is an invariant of the Tukey order and therefore of order isomorphism.
An equivalent condition for $\lambda\in \spec_T(\mathbb{D})$ is to require the existence of a sequence $\l d_\alpha\mid \alpha<\lambda\r$ such that for every $I\in [\lambda]^\lambda$, $\{d_\alpha\mid \alpha\in I\}$ is unbounded in $\mathbb{D}$.
Corollary \ref{cor: comparison in generic} has the following corollary:
\begin{corollary}
     Let $\mathbb{P}$ be either $\kappa^+$-cc or $|\mathbb{D}|^+$-closed forcing and $\mathbb{D}$ a $\kappa^+$-directed set. Let $G\s \mathbb{P}$ be a $V$-generic filter, then  $$\spec_T(\mathbb{D})^V=\spec_T(\mathbb{D})^{V[G]}.$$
\end{corollary}
\begin{proof}
    Note that $\spec_T(\mathbb{D})\subseteq [\kappa^+,|\mathbb{D}|]$, hence Corollary \ref{cor: comparison in generic} implies the result for $\kappa^+$-cc forcings instantly. If $\mathbb{P}$ is $|\mathbb{D}|^+$-closed then this is much simpler since all the witnessing sequences for $\lambda$'s being in the Tukey spectrum exist in the ground model. 
\end{proof}
For the following fact, see for example \cite{Gilton}:
\begin{fact}\label{fact:Easy fact} {}\
   \begin{enumerate}
       \item For any directed sets $\mathbb{P},\mathbb{Q}$, $\spec_T(\mathbb{P}\times \mathbb{Q})=\spec_T(\mathbb{P})\cup \spec_T(\mathbb{Q})$.
       \item If $\mathbb{P}$ is linear, then $\spec_T(\mathbb{P})=\{\cf(\mathbb{P})\}$.
   \end{enumerate}  
\end{fact}
The main results of the previous sections have the following immediate corollaries:
\begin{corollary}
   Suppose that $\mathbb{D}$ is a well-founded, countably directed set with $|\mathbb{D}|=\mu$. Then there is a ($\mu^+$-closed $\ast$ ccc) forcing extension in which there is an ultrafilter $U$ on $\omega$ such that $U\equiv_T \mathbb{D}$ and in particular
$\spec_T(U)=\spec_T(\mathbb{D}).$
\end{corollary}
\begin{corollary}
   Suppose that $\kappa$ is a supercompact cardinal and let $\mathbb{D}$ be a $\kappa^+$-directed well-founded set with $|\mathbb{D}|=\mu$. Then there is a ($\kappa$-cc $\ast \ \mu^+$-closed $\ast \ \kappa^+$-cc) forcing extension in which there is a normal ultrafilter $U$ on $\kappa$ such that $U\equiv_T \mathbb{D}$, and in particular
$\spec_T(U)=\spec_T(\mathbb{D}).$
\end{corollary}
 As for every normal ultrafilter $U$, the order $(U, \supseteq^*)$ is $\kappa^{+}$-directed, this result is locally optimal --- every partial order that might be compatible with the order $\supseteq^*$ of a normal ultrafilter can consistency be forced to actually be Tukey equivalent to one.

The next proposition shows that for the point spectrum to be preserved, we do not need any directedness  assumptions. More specifically:
\begin{proposition}
    Let $\mathbb{P}$ be a $\lambda$-cc forcing and $\mathbb{D}$ a directed set. Let $G\s \mathbb{P}$ be a $V$-generic filter, then  $$\lambda\in\spec_T(\mathbb{D})^V\text{ if and only if }\lambda\in\spec_T(\mathbb{D})^{V[G]}.$$
\end{proposition}
\begin{proof}
    The implication from left to right always hold true.
     For the converse, suppose that $\lambda\in \spec_T(\mathbb{D})^{V[G]}$. Let $\langle d_\alpha\mid \alpha<\lambda\rangle\subseteq D$ be a sequence in $V[G]$ such that for each index set $I\in [\lambda]^\lambda$ the sequence $\langle d_\alpha\mid \alpha\in I\rangle$ is unbounded.  Let $p\in G$ be a condition forcing the above. For each $\alpha<\lambda$ let $\dot{d}_\alpha$ be a $\mathbb{P}$-name such that $(\dot{d}_\alpha)_G=d_\alpha$. For each $\alpha<\lambda$ let $p_\alpha \leq p$ such that $p_\alpha\Vdash_{\mathbb{P}} \dot{d}_\alpha=\check{e}_\alpha$ for some $e_\alpha\in D$. We claim that $\langle e_\alpha\mid \alpha<\lambda\rangle$ witnesses $\lambda\in \spec(\mathbb{D})^V$. Suppose otherwise, and let $I\in [\lambda]^\lambda$ be such that $\langle e_\alpha\mid \alpha\in I\rangle$ is bounded. To contradict this assumption it suffices to check that $p\Vdash_{\mathbb{P}}``|\{\alpha<\lambda\mid p_\alpha\in \dot{G}\}|=\lambda"$. So, let us do it: Suppose towards a contradiction that this was false. Since $\mathbb{P}$ is $\lambda$-cc there is $\beta<\lambda$ such that $p\Vdash \{\alpha<\lambda\mid p_\alpha\in \dot{G}\}\s \beta$. Note that this is impossible because $p_\beta$ itself forces $``p_\beta\in \dot{G}"$.
\end{proof}

A related notion is the depth spectrum introduced in \cite{TomCohesive}:
\begin{definition}\label{Def: Depth Spec}
   Let $\mathbb{D}=(D,\leq_D)$ be a directed set. A \textit{$\mathbb{D}$-tower of length $\lambda$} is a sequence $\l d_\alpha\mid \alpha<\lambda\r$ which is $\leq_D$-increasing and unbounded in $\mathbb{D}$. The \textit{depth spectrum of} $\mathbb{D}$, denoted by $\spec_{dp}(\mathbb{D})$ consists of all regular cardinals $\lambda$ such that there is a $\mathbb{D}$-tower of length $\lambda$.
\end{definition}
The following proposition is straightforward:
\begin{proposition}\label{Prop: Tukey reductions}
     Let $\mathbb{D}_i=(D_i,\leq_{D_i})$, $i=1,2$ be directed sets. If there is a monotone Tukey reduction $f:D_1\to D_2$ then $\spec_{dp}(D_1)\subseteq \spec_{dp}(D_2)$.
\end{proposition}
\begin{remark}\label{remark: basis remark}
    If $\mathcal{B}$ is cofinal in $\mathbb{D}=(D,\leq_D)$, then\footnote{Formally, $\mathbb{D}\equiv_T(\mathcal{B},\leq_D\restriction \mathcal{B})$.}, $\mathbb{D}\equiv_T \mathcal{B}$. Moreover, the identity map from $\mathcal{B}$ into $D$ is a monotone Tukey reduction. 
\end{remark}
Our interest is in orders of the form $(U,\supseteq^*)$ where $U$ is an ultrafilter.
Recall that the decomposability spectrum of an ultrafilter $U$ over $\kappa$, denoted by $\spec_{dc}(U)$, is defined as all regular $\lambda$ such that $U$ is $\lambda$-decomposable, that is, there is $f:\kappa\to \lambda$ such that $f$ is unbounded in $\lambda$ mod $U$. We have that
$$\spec_{dc}(U)\setminus \{\kappa\}\subseteq \spec_{dp}(U)\setminus \{\kappa\}\subseteq \spec_T(U) \setminus \{\kappa\}$$
The following was asked in \cite{TomCohesive}:
\begin{question}
    Is $\spec_T(U)$ a convex set of regulars? How about $\spec_{dp}(U)$? 
\end{question}
Using the consistency results from the previous sections and the next proposition we are able to provide a negative answer to this question.
\begin{proposition}\label{prop: Spctrum of simple point}
    Suppose that $\mathcal{F}$ is a simple $P_{\lambda_0\times...\times\lambda_n}$-point filter, then $\spec_{dp}(U)=\spec_{T}(U)=\{\lambda_0,...,\lambda_n\}$.
\end{proposition}
\begin{proof}
    Let $\mathcal{B}$ be a base for $U$ such that $(\mathcal{B},\supseteq^*)\simeq \lambda_0\times \lambda_1\times...\times\lambda_n$. Then $$U\equiv_T\lambda_0\times\lambda_1\times...\times\lambda_n,$$
    and by Fact \ref{fact:Easy fact},  $\spec_T(U)=\{\lambda_0,...\lambda_n\}$. By Remark \ref{remark: basis remark} and Proposition \ref{Prop: Tukey reductions}, $\spec_{dp}(\lambda_0\times...\times \lambda_n)\subseteq\spec_{dp}(U,\supseteq^*)$. Finally note that for every $0\leq i\leq n$, $\lambda_i\in \spec_{dp}(\lambda_0\times...\times\lambda_n)$ as witnessed by  the sequence $$\langle (0,0,...,\underset{i^{\text{th}}\text{place}}{\underbrace{\alpha}},...0,0)\mid \alpha<\lambda_i\rangle.$$
     
    Putting all of the above together, we have
    \begin{align*}
        \{\lambda_0,\lambda_1,...,\lambda_n\}&\subseteq \spec_{dp}(\lambda_0\times...\times\lambda_n)\\
        &\subseteq \spec_{dp}(U,\supseteq^*)\\
        &\subseteq \spec_T(U,\supseteq^*)=\{\lambda_0,\lambda_1,...,\lambda_n\}
    \end{align*}
     
\end{proof}
\begin{corollary}\label{Thm: nonConvex}\hfill
    \begin{enumerate}
        \item Let $\omega<\lambda_0<\lambda_1 < \lambda_2 < \cdots < \lambda_n$ be regular cardinals. It is consistent that there is an ultrafilter on $\omega$, such that  $\spec_{dp}(U)=\spec_T(U)=\{\lambda_0,\lambda_1, \lambda_2,\dots, \lambda_n\}$.
        \item Suppose that $\kappa$ is supercompact. Then for any regular cardinals $\kappa<\lambda_0<\lambda_1<...<\lambda_n$, there is a ($\kappa$-cc $\ast \ \lambda_n^+$-closed $\ast \ \kappa^+$-cc) extension where:
        \begin{enumerate}
            \item There is a normal ultrafilter $U$ such that $\spec_{dp}(U)=\spec_T(U)=\{\lambda_0,\lambda_1,...\lambda_n\}$.
            \item $\spec_{dp}(\text{Cub}_\kappa)=\spec_T(\text{Cub}_\kappa)=\{\lambda_0,\lambda_1...,\lambda_n\}$.
        \end{enumerate}
    \end{enumerate} 
\end{corollary}
\begin{proof}
    Using Proposition \ref{prop: Spctrum of simple point}, $(1),(2a),(2b)$ follow from  Theorem \ref{thm: the omega simple p-point}, \ref{Cor: Simple P-point measurable}, and Corollary \ref{Cor: Clubs} respectively.
\end{proof}

Gitik \cite{GitikNew} notified us that he was able to obtain $(2b)$ independently for the depth spectrum of the club filter at a measurable cardinal from optimal assumptions.
Next, let us consider the Tukey-related notion of cohesiveness due to Kanamori \cite{Kanamori1978}. This notion is also known as Galvin's property \cite{bgp,TomGabe24,bgs,SatInCan}:
\begin{definition}\label{def:cohesive}
    An ultrafilter $U$ is {\em$(\mu,\lambda)$-cohesive} if for any $\{X_\alpha\mid \alpha<\lambda\}\in [U]^\lambda$ there is $I\in [\lambda]^\mu$ such that $\bigcap_{i\in I}X_i\in U$.
\end{definition}
Thus, $\lambda\in \spec_T(U)$ if and only if $U$ is \textit{not} $(\lambda,\lambda)$-cohesive (see \cite{TomCohesive}).
The ultrafilters constructed in this paper can also be used to separate the notion of $(\lambda_0,\lambda_1)$-cohesive from being $(\lambda_0,\lambda_0)$ and $(\lambda_1,\lambda_1)$-cohesive. It is easy to see that any non $(\lambda_0,\lambda_1)$-cohesive ultrafilter is not  $(\lambda_0,\lambda_0)$ and not $(\lambda_1,\lambda_1)$-cohesive. The converse is not true in general, as witnessed by our ultrafilters:
\begin{theorem}
    If $U$ is a simple  $P_{\lambda_0\times\lambda_1}$-point over $\kappa$ for regular cardinals $\kappa<\lambda_0<\lambda_1$, then $U$ is not $(\lambda_0,\lambda_0)$ and $(\lambda_1,\lambda_1)$-cohesive but it is $(\lambda_0,\lambda_1)$-cohesive.
\end{theorem}
\begin{proof}
    As mentioned in the paragraph following Definition \ref{def:cohesive}, not being $(\lambda,\lambda)$-cohesive is equivalent to $\lambda\in \spec_T(U)$. Hence the first part follows from Proposition \ref{prop: Spctrum of simple point}. For the second part, let us prove that $U$ is $(\lambda_0,\lambda_1)$-cohesive. Fix $\l X_i\mid i<\lambda_1\r\subseteq U$. We need to find $\lambda_0$-many sets whose intersection is in $U$. Let $\mathcal{B}=(B_{i,j})_{(i,j)\in\lambda_0\times\lambda_1}$ be a base for $U$ witnessing that $U$ is a simple $P_{\lambda_0\times\lambda_1}$-point. For each $i<\lambda_1$ there is $\beta_i<\lambda_1$ and $\alpha_i<\lambda_0$ such that $B_{\alpha_i,\beta_i}\subseteq^* X_i$. There is $I\in [\lambda_1]^{\lambda_1}$  and $\alpha^*<\lambda_0$ such that for every $i\in I$, $\alpha_i=\alpha^*$. Consider the first $\lambda_0$-many indices $\{i_\gamma\mid\gamma<\lambda_0\}\subseteq I$. Let $i^*=\sup_{\gamma<\lambda_0}i_\gamma+1<\lambda_1$. Then $B_{\alpha^*,i^*}\subseteq^* B_{\alpha^*,i_{\gamma}}\subseteq ^*X_{i_\gamma}$ for every $\gamma<\lambda_0$. Since $\lambda_0>\kappa$ is regular, we can find $J\subseteq \{i_\gamma\mid \gamma<\lambda_0\}$ still of size $\lambda_0$ and some $\xi<\kappa$ such that $B_{\alpha^*,i^*}\setminus \xi\subseteq \bigcap_{j\in J}X_j$, as wanted.
\end{proof}
As in the case of the linear iteration of the Mathias forcing, our model can exhibit a small ultrafilter number.
\begin{proposition}\label{prop: small ultrafilter number}
    Suppose that $\kappa$ is a supercompact cardinal and let $\mathbb{D}$ be a $\kappa^+$-directed well-founded set with $|\mathbb{D}|\leq\lambda$. Then there is a ($\kappa$-cc $\ast \ \lambda^+$-closed $\ast \ \kappa^+$-cc) generic extension with a simple $P_{\mathbb{D}}$-point ultrafilter and $2^\kappa=\lambda$. 
\end{proposition}
\begin{proof}
    Let $S^*_{\dot{U}}\subseteq\lambda^+$, $(B_{a})_{ a\in S^*_{\dot{U}}\times\mathbb{D}}$, and $U$ be as in Theorem \ref{msble}. Arguing as in the proof \ref{Cor: Simple P-point measurable}, we choose $ s^*\in S^*_{\dot{U}}$ a limit point of $S^*_{\dot{U}}$ with an appropriate cofinal subset included in $S^*_{\dot{U}}$. This time we make sure that $s^*\geq \lambda$. Let
     $$U^*=U\cap V[G_{s^*\times \mathbb{D}}]\in V[G_{s^*\times \mathbb{D}}]=V^*$$
    The ultrafilter $U^*$ is a simple $P_{\mathbb{D}}$-point. Also, since $s^*\geq\lambda$, in $V^*$, $2^\kappa=\lambda$, as wanted.   
\end{proof}
In the following theorem we use the well-known characterization of the generalized bounding and dominating numbers using the club filter (see for example \cite{TomGabe2025}):
\begin{proposition}
    Let $\kappa$ be a regular cardinal. Then
    $$\mathfrak{d}_\kappa=\min\{|\mathcal{B}|\mid \mathcal{B} \ {\subseteq^*}\text{-generates }\text{Cub}_{\kappa}\}$$
        $$\mathfrak{b}_\kappa=\min\{|\mathcal{B}|\mid \mathcal{B} \text{ is unbounded in }(\text{Cub}_{\kappa},\supseteq^*)\}$$
\end{proposition}
Note that being unbounded in $(\text{Cub}_\kappa,\supseteq^*)$ is equivalent to not having a pseudo-intersection (which is just unbounded in $\kappa$).
\begin{theorem}
    Relative to a supercompact cardinal, it is consistent that $\kappa$ is a supercompact cardinal and $\mathfrak{b}_\kappa<\mathfrak{d}_\kappa=\mathfrak{u}_\kappa<2^\kappa$
\end{theorem}
\begin{proof}
    Let $V^*$ be the model of Proposition \ref{prop: small ultrafilter number} where $2^\kappa=\lambda_2>\lambda_1>\lambda_0$ and there is a simple $P_{\lambda_0\times\lambda_1}$-point over $\kappa$. By Corollary \ref{Cor: Clubs} the club filter is a simple $P_{\lambda_0\times\lambda_1}$-point and in particular generated by $\lambda_1$-many sets. By the previous proposition, in $V^*$, $\mathfrak{d}_\kappa=\lambda_1$. Also, any simple $P_{\lambda_0\times\lambda_1}$-point has an unbounded family of size $\lambda_0$ and therefore by the previous proposition $\mathfrak{b}_\kappa=\lambda_0$.  In ZFC, for uncountable cardinals  we have the following inequalities (see for example \cite{TomGabe2025}): $$\mathfrak{d}_\kappa\leq \mathfrak{u}_\kappa\leq \mathfrak{u}^{com}_\kappa.$$ Since $U^*$ is also  generated by $\lambda_1$-many sets, we have $\mathfrak{u}^{com}_\kappa\leq\lambda_1$ and we get that  $\mathfrak{d}_\kappa=\mathfrak{u}_\kappa=\mathfrak{u}^{com}_{\kappa}=\lambda_1$. 
\end{proof} 
Let us conclude this paper with two remarks:
\begin{remark}
     The above model is essentially different from the linear Mathias iteration in the sense that a $P_\lambda$-point cannot exist in those models: indeed, if there is a simple $P_\lambda$-point then $\mathfrak{b}_\kappa=\mathfrak{d}_\kappa$ (see \cite{TomGabe2025}), which is not going to hold here.
\end{remark}
\begin{remark}
    The existence of a simple $P_{\mathbb{D}}$-point ultrafilter $U$ poses restrictions on cardinal characteristics. For example, by \cite{TomGabe2025} on a measurable cardinal $\kappa$, $\min(\spec_T(\mathbb{D}))=\min(\spec_T(U))\leq\mathfrak{b}_\kappa$ and $\mathfrak{ch}(\mathbb{D})=\mathfrak{ch}(U)\geq \mathfrak{d}_\kappa$. However, it is unclear if this equality must hold or if there are any other limitations (See Questions \ref{question: limitations}-\ref{question: limitations3}). 
\end{remark}

\section{Problems}
\begin{question}
    What kind of ill-founded directed sets $\mathbb{D}$ can be realized as simple $P_{\mathbb{D}}$-point ultrafilters?
\end{question}
\begin{question}
Is there a method to realize a poset $\mathbb{D}$ as a simple $P_{\mathbb{D}}$-point ultrafilter for $\mathbb{D}$'s which are not countably-directed? 
\end{question}
Some limitations must be placed as by Corollary \ref{cor: limitation} no simple $P_{\omega\times \omega_1}$-point exists. But what are the exact limitations?

In this paper, we focused on $\subseteq^*$. Still, one might also be interested in the order $\supseteq$ on an ultrafilter $U$, especially when considering the Tukey-type of $(U,\supseteq)$. 
\begin{question}
    What are the possible isomorphism types of $(\mathcal{B},\supseteq)$, where $\mathcal{B}$ is a $\subseteq$-generating set of an ultrafilter?
\end{question}
\begin{question}
    Is it consistent that $\mathfrak{d}_\kappa<\mathfrak{u}_\kappa<2^\kappa$ for a measurable cardinal $\kappa$?
\end{question}
Finally, we would like to ask about the possible limitations that the existence of a simple $P_{\mathbb{D}}$-point poses. 
\begin{question}\label{question: limitations}
    Let $\kappa$ be a measurable cardinal. Can there be $\kappa$-complete simple $P_{\mathbb{D}_0}$-point and a simple $P_{\mathbb{D}_1}$-point ultrafilters on $\kappa$ for two posets $\mathbb{D}_0,\mathbb{D}_1$ such that $\mathbb{D}_0\not\equiv_T\mathbb{D}_1$?
\end{question}
For linear orders $\mathbb{D}_0,\mathbb{D}_1$ this was proven to be impossible in \cite{TomGabe2025}.
\begin{question}\label{question: limitations2}
    Let $\kappa$ be a measurable cardinal and suppose that there is a $\kappa$-complete simple $P_{\mathbb{D}}$-point ultrafilter on $\kappa$. Is any of the following statements provable:
    \begin{enumerate}
        \item $\mathfrak{b}_\kappa=\min(\spec_T(\mathbb{D}))$.
        \item $\mathfrak{d}_\kappa=\mathfrak{d}(\mathbb{D})$. Here $\mathfrak{d}(\mathbb{D})$ denotes the minimal size of a cofinal subset of $\mathbb{D}$.
        \item $\spec_T(\mathbb{D})=\spec_T(\text{Cub}_\kappa)$.
    \end{enumerate}
\end{question}
Answering the question for a specific non-linear $\mathbb{D}$ (e.g. $\mathbb{D}=\lambda_0\times\lambda_1$) would also be of interest.

On $\omega$, \cite{MildenBraun} showed that it is consistent to have a simple $P_{\aleph_1}$-point and a simple $P_{\aleph_2}$-point. Nyikos \cite{Nyikos2020-NYISUA} proved that if there is a simple $P_\lambda$-point on $\omega$ then either $\lambda=\mathfrak{b}$ or $\lambda=\mathfrak{d}$. The following questions relate to generalizations of these facts:
\begin{question}\label{question: limitations3}
    Can there be $\mathbb{D}_0,\mathbb{D}_1,\mathbb{D}_2$ non-Tukey equivalent such that there are simple $P_{\mathbb{D}_i}$-point ultrafilters on $\omega$ for $i=0,1,2$?
\end{question}

\bibliographystyle{amsplain}
\bibliography{ref1}

\end{document}